  \documentclass[11pt,letterpaper]{amsart}
\makeatletter
 \def\@textbottom{\vskip \z@ \@plus 1pt}
 \let\@texttop\relax
\makeatother
\usepackage{amsxtra}
\usepackage[all]{xy}
\usepackage{palatino}
\usepackage{commath}
\usepackage{mathrsfs}
\usepackage{comment}

\usepackage{tikz-cd}
\usepackage[bookmarks=false]{hyperref}

\usepackage{amssymb,amsmath,amsthm,epsf,epsfig,dsfont,bbm}
\usepackage{upref, eucal}

\numberwithin{equation}{section}
\usepackage{amssymb,amsmath,amsthm,epsf,epsfig,dsfont,bbm}
\usepackage{latexsym}
\usepackage{mathtools}
\usepackage{xcolor}

\DeclareMathOperator{\Spec}{\mathrm{Spec}}

\DeclareUnicodeCharacter{0301}{\'{e}}

\newcommand{\beqar}{\begin{eqnarray*}}
\newcommand{\eeqar}{\end{eqnarray*}}

\newcommand{\oldmarginpar}[1]{}

\begin{document}

\newcommand\J{\mathfrak{J}}
\newcommand\A{\mathbb{A}}
\newcommand\C{\mathbb{C}}
\newcommand\G{\mathbb{G}}
\newcommand\N{\mathbb{N}}
\newcommand{\T}{\mathbb{T}}
\newcommand{\cT}{\mathcal{T}}
\newcommand\sO{\mathcal{O}}
\newcommand\sE{{\mathcal{E}}}
\newcommand\tE{{\mathbb{E}}}
\newcommand\sF{{\mathcal{F}}}
\newcommand\sG{{\mathcal{G}}}
\newcommand\GL{{\mathrm{GL}}}
\newcommand{\HH}{\mathrm H}
\newcommand\mM{{\mathrm M}}
\newcommand\fS{\mathfrak{S}}
\newcommand\fP{\mathfrak{P}}
\newcommand\fQ{\mathfrak{Q}}
\newcommand\Qbar{{\bar{\Q}}}
\newcommand\sQ{{\mathcal{Q}}}
\newcommand\sP{{\mathbb{P}}}
\newcommand{\Q}{\mathbb{Q}}
\newcommand{\tH}{\mathbb{H}}
\newcommand{\Z}{\mathbb{Z}}
\newcommand{\R}{\mathbb{R}}
\newcommand{\F}{\mathbb{F}}
\newcommand{\D}{\mathfrak{D}}
\newcommand\gP{\mathfrak{P}}
\newcommand\Gal{{\mathrm {Gal}}}
\newcommand\SL{{\mathrm {SL}}}

\newcommand{\legendre}[2] {\left(\frac{#1}{#2}\right)}
\newcommand\iso{{\> \simeq \>}}
\newcommand{\cX}{\mathcal{X}}
\newcommand{\cJ}{\mathcal{J}}
\newtheorem{thm}{Theorem}
\newtheorem{theorem}[thm]{Theorem}
\newtheorem{cor}[thm]{Corollary}
\newtheorem{conj}[thm]{Conjecture}
\newtheorem{prop}[thm]{Proposition}
\newtheorem{lemma}[thm]{Lemma}
\theoremstyle{definition}
\newtheorem{definition}[thm]{Definition}
\theoremstyle{remark}
\newtheorem{remark}[thm]{Remark}
\newtheorem{example}[thm]{Example}
\newtheorem{claim}[thm]{Claim}
\newtheorem{lem}[thm]{Lemma}

\theoremstyle{definition}
\newtheorem{dfn}{Definition}

\theoremstyle{remark}

\theoremstyle{remark}
\newtheorem*{fact}{Fact}
% type user-defined commands here
\makeatletter
\def\imod#1{\allowbreak\mkern10mu({\operator@font mod}\,\,#1)}
\makeatother
\newcommand{\mcF}{\mathcal{F}}
\newcommand{\mcG}{\mathcal{G}}
\newcommand{\mcM}{\mathcal{M}}
\newcommand{\mcO}{\mathcal{O}}
\newcommand{\mcP}{\mathcal{P}}
\newcommand{\mcS}{\mathcal{S}}
\newcommand{\mcV}{\mathcal{V}}
\newcommand{\mcW}{\mathcal{W}}
\newcommand{\gS}{\mathcal{S}}
\newcommand{\cO}{\mathcal{O}}
\newcommand{\cC}{\mathcal{C}}
\newcommand{\mfa}{\mathfrak{a}}
\newcommand{\mfb}{\mathfrak{b}}
\newcommand{\mfH}{\mathfrak{H}}
\newcommand{\mfh}{\mathfrak{h}}
\newcommand{\mfm}{\mathfrak{m}}
\newcommand{\mfn}{\mathfrak{n}}
\newcommand{\mfp}{\mathfrak{p}}
\newcommand{\mfq}{\mathfrak{q}}
\newcommand{\cF}{\mathfrak{F}}
\newcommand{\mrB}{\mathrm{B}}
\newcommand{\mrG}{\mathrm{G}}
\newcommand{\gG}{\mathcal{G}}
\newcommand{\mrH}{\mathrm{H}}
\newcommand{\mH}{\mathrm{H}}
\newcommand{\mrZ}{\mathrm{Z}}
\newcommand{\Ga}{\Gamma}
\newcommand{\cyc}{\mathrm{cyc}}
\newcommand{\Fil}{\mathrm{Fil}}
\newcommand{\mrp}{\mathrm{p}}
\newcommand{\PGL}{\mathrm{PGL}}
\newcommand{\x}{{\mathcal{X}}}
\newcommand{\Sp}{\textrm{Sp}}
\newcommand{\ab}{\textrm{ab}}

\newcommand{\lra}{\longrightarrow}
\newcommand{\ra}{\rightarrow}
\newcommand{\rai}{\hookrightarrow}
\newcommand{\ras}{\twoheadrightarrow}

\newcommand{\repr}{\rho_{f,\wp}|_{G_p}}
\newcommand{\GRF}{{\rho}_{f,\wp}}

\newcommand{\lan}{\langle}
\newcommand{\ran}{\rangle}

\newcommand{\mo}[1]{|#1|}

\newcommand{\hw}[1]{#1+\frac{1}{2}}
\newcommand{\mcal}[1]{\mathcal{#1}}
\newcommand{\trm}[1]{\textrm{#1}}
\newcommand{\mrm}[1]{\mathrm{#1}}
\newcommand{\car}[1]{|#1|}
\newcommand{\pmat}[4]{ \begin{pmatrix} #1 & #2 \\ #3 & #4 \end{pmatrix}}
\newcommand{\bmat}[4]{ \begin{bmatrix} #1 & #2 \\ #3 & #4 \end{bmatrix}}
\newcommand{\pbmat}[4]{\left \{ \begin{pmatrix} #1 & #2 \\ #3 & #4 \end{pmatrix} \right \}}
\newcommand{\psmat}[4]{\bigl( \begin{smallmatrix} #1 & #2 \\ #3 & #4 \end{smallmatrix} \bigr)}
\newcommand{\bsmat}[4]{\bigl[ \begin{smallmatrix} #1 & #2 \\ #3 & #4 \end{smallmatrix} \bigr]}

\makeatletter
\def\imod#1{\allowbreak\mkern10mu({\operator@font mod}\,\,#1)}
\makeatother
\title{The Heisenberg covering of the Fermat curve}

\author{Debargha Banerjee}
\address{Indian Institute of science, education and research, Pune, India}
\author{Lo\"ic Merel}
\address{Universit\'e Paris Cit\'e and Sorbonne Universit\'e, CNRS, IMJ-PRG, F-75013 Paris, France.}
\thanks{The author was partially supported by the SERB grant MTR/2017/000357 and CRG/2020/000223. The first named author is deeply indebted to Professor Yuri Manin for several stimulating conversation at the MPIM}
\begin{abstract}
For $N$ integer $\ge1$, K. Murty and D. Ramakrishnan defined the $N$-th Heisenberg curve, as the compactified quotient $X'_N$ of the upper half-plane by a certain non-congruence subgroup of the modular group. They ask whether the Manin-Drinfeld principle holds, namely if the divisors supported on the cusps of those curves are torsion in the Jacobian. We give a model over ${\bf Z}[\mu_N,1/N]$ of  the $N$-th Heisenberg curve as covering of the $N$-th Fermat curve. We show that the Manin-Drinfeld principle holds for $N=3$, but not for $N=5$. We show that the description by generator and relations due to Rohrlich of the cuspidal subgroup of the Fermat curve is explained by the Heisenberg covering, together with a higher covering of a similar nature.
The curves $X_N$ and the classical modular curves $X(n)$, for $n$ even integer, both dominate $X(2)$, which produces a morphism between jacobians $J_N\rightarrow J(n)$. We prove that the latter has image $0$ or an elliptic curve of $j$-invariant $0$. In passing, we give a description of the homology of $X'_{N}$.
 \end{abstract}

\subjclass[2010]{Primary: 11D11, Secondary: 11F11, 11G05, 11G30}
\keywords{Fermat's curves,  Modular symbols, Heisenberg curves}
\maketitle
 \setcounter{tocdepth}{1}
\tableofcontents{}

\section{Introduction}
\label{intro}
Let $\Ga$ be a subgroup of finite index of $\SL_2(\Z)$. This subgroup acts by homographies on the complex upper half-plane $\mfH$. Consider the corresponding modular curve $Y_\Ga=\Ga\backslash \mfH$, and its compactification obtained by adding the cusps $X_\Ga$. We say that $X_\Ga$ {\it satisfies the Manin-Drinfeld} principle if any cuspidal ({\it i.e.} supported on the cusps) divisor of degree $0$ is torsion in the Jacobian of $X_\Ga$. Manin and Drinfeld proved that it is the case when $\Ga$ is a congruence subgroup.

For a subgroup of finite index (not necessarily a congruence subgroup), K. Murty and Ramakrishnan  ~\cite{MR983619} give an analytic criterion for the Manin-Drinfeld principle to be satisfied. 
As an illustrative example, Murty and Ramakrishnan consider modular curves attached to certain subgroups of $\Ga(2)$: Fermat curves, and what they propose to call Heisenberg curves. We revisit those examples. 
In \cite{BanerjeeMerel}, we reconsider this question and give an analytic criterion of a different nature, but also based on Eisenstein series ; this is unconnected to the present work, which is purely algebraic in nature.

The Heisenberg curves are defined as follows from the complex analytic point of view. Let $A$ and $B$ be the classes of the matrices $\begin{pmatrix}
 1 & 2 \\
 0 & 1 \\
\end{pmatrix}$ and $\begin{pmatrix}
 1 & 0 \\
 2 & 1 \\
\end{pmatrix}$ 
respectively in 
${\bar\Ga}(2)=\Ga(2)/\{\pm1\}$. These matrices generate freely the group ${\bar\Ga}(2)$. Let $C=ABA^{-1}B^{-1}$.

Let $N$ be an integer $>0$. Denote by $\Phi_N$ the kernel of the morphism ${\bar\Ga}(2)\rightarrow(\Z/N\Z)^2$ which to $A$ associates $(1,0)$ and to $B$ associates $(0,1)$. The corresponding modular curve is the {\it Fermat modular curve} and is denoted by $X_{N}$. It can be identified with the complex points of the Fermat curve $F_N$ (see for instance ~\cite{MR0441978}, ~\cite{MR681120}). Let $\Phi_N'$ be the subgroup of $\Phi_N$ generated by $A^N$, $B^N$, $C^N$ and the third term $[{\bar\Ga}(2),[{\bar\Ga}(2),{\bar\Ga}(2)]]$ in the descending central series of ${\bar\Ga}(2)$.  An exact sequence of groups follows
$$
1\rightarrow \Phi_N'\rightarrow {\bar\Ga}(2)\rightarrow H_N\rightarrow 1,
$$
where $H_N$ is a central extension of $({\Z/N\Z})^{2}$ by ${\Z/N\Z}$ (coinciding with the $\Z/N\Z$-points of the Heisenberg group).
The {\it $N$-th Heisenberg modular curve}, in the sense of Murty and Ramakrishnan, is $X'_N=X_{\Phi'_N}$.

Let $\Q(\mu_{N})$ be a cyclotomic extension of $\Q$ generated by $N$-th roots of unity. Denote by $\Z[\mu_{N}]$ its ring of integers.
The covering $X'_{N}\rightarrow X_{N}$ extends to a morphism $F'_{N}\rightarrow F_{N}$ of curves over $\Q(\mu_{N})$ that we call the {\it Heisenberg covering of the Fermat curve}. 

\begin{theorem}
\label{Modelthm}
Suppose $N$ is an odd integer. The Heisenberg modular curve $X'_N$ extends to a smooth projective scheme ${\mathcal{F}}'_N$ of relative dimension one over ${\rm Spec}({\Z}[\mu_{N},1/N])$ given by the following model: 
$$
X^{N}+Y^{N}=Z^{N}
$$ 
and, for every primitive $N$-th root of unity $\zeta$ in $\Q(\mu_{N})$
$$
\prod_{j=1}^{(N-1)/2}(Y-\zeta^{-j}Z)^{j}T_{\zeta}^{N}=\prod_{j=1}^{(N-1)/2}(Y-\zeta^{j}Z)^{j}U_{\zeta}^{N}.
$$

\end{theorem}

It seems to have been known to Deligne (see a comment in ~\cite{MR983619}) that the generic fiber  $F'_{N}$ of ${\mathcal{F}}'_N$ can be defined over ${\Q}$, an assertion for which we provide a proof. 

Rohrlich ~\cite{MR0441978} (see also V\'elu ~\cite{MR582434}) has determined the cuspidal subgroup of $F_N$. In particular, he has shown that any cuspidal divisor on $F_N$ is of order dividing $N$. This description plays a key role in justifying the existence of the Heisenberg covering. 
We show that, by going further in the descending central series of $\Ga(2)$, $X'_N$ is covered by a modular curve $X''_N$, in such a way that $X''_N$ is still an abelian covering of the Fermat curve $X_{N}$. We do not describe algebraically $X''_N$.

We note that there has been a considerable interest in the cuspidal group of the Fermat curve. For instance, in ~\cite{MR488287}, p. 39, Mazur draws (or rather ``stretches'') an analogy between Fermat curves and modular curves. Such an analogy is somewhat strengthened by the fact that the Heisenberg covering is analogous to the familiar Shimura covering $X_1(N)\rightarrow X_0(N)$ between modular curves.

Like Murty and Ramakrishnan, our goal had been to illustrate our study of non-congruence subgroups by examining Heisenberg curves. 
We can  not determine in general whether such curves satisfy the Manin-Drinfeld principle. 
But we can show easily that the principle holds for $N=3$. Furthermore, for $N=3$, we study the connection between $F'_{3}$ and various modular curves.
By contrast, for $N=5$:

\begin{theorem}
\label{theoremN=5}
There exists a cuspidal divisor on $X'_5$ whose class in the jacobian of $X'_5$ is of infinite order.
\end{theorem}

Let $\bar\Gamma'(2)$ be the subgroup of index $3$ of $\bar\Gamma(2)$ obtained by pulling back the $2$-Sylow subgroup of ${\rm PSL}_2(\Z/3\Z)$.
Let $\Gamma$ be a congruence subgroup of $\bar\Gamma(2)$.
Consider the correspondence $X'_N\rightarrow X_\Gamma$ obtained by combining pulling back to the modular curve $X_{\Gamma\cap \Phi'_N}$ with pushing to $X_\Gamma$.
It produces a morphism of abelian varieties between the jacobians or Riemann surfaces $\theta_{N,\Gamma}$ : $J'_N\rightarrow J_\Gamma$. 
In view of the following statement, we can hardly have any hope of establishing a limited form of the Manin-Drinfeld principle for Heisenberg curves using the classical Manin-Drinfeld theorem for congruence subgroups.

\begin{thm}
\label{theoremcomparison}
The morphism $\theta_{N,\Gamma}$ is zero if and only if either $3 \nmid N$ or $\Gamma$ is not contained in $\bar\Gamma'(2)$.
If $3 \mid N$ and $\Gamma$ is contained in $\bar\Gamma'(2)$, the image of $\theta_{N,\Gamma}$ is isogenous to an elliptic curve with $j$-invariant $0$. Furthermore, when $\Gamma$ is contained in $\bar\Gamma'(2)$, $\theta_{3,\Gamma}$ has finite kernel.
\end{thm}

The proof of theorem \ref{theoremcomparison} is a translation of group theoretic statement: any term the lower central series of $\bar\Ga(2)$ is essentially dense in adelic completions of $\bar\Ga(2)$ (see Proposition \ref{oddadelic}).
In addition to these results of algebraic nature, we give a combinatorial description of the homology of the Riemann surface $X'_{N}$, by a method similar to Manin's presentation, but following the variant introduced in \cite{MR1405312}. This might have an interest of its own. But it does not help us for establishing our other results.
It will be intriguing to develop the theory of modular symbols on the Heisenberg curves following Shokurov~ \cite{MR0563089}, \cite{MR582162},\cite{MR571104}. 

\section{Heisenberg Groups and the lower central series of $\bar\Ga(2)$}
\label{}

\subsection{The Heisenberg group}
The Heisenberg group is the algebraic group of $3\times 3$ unipotent upper triangular matrices. Set:
 $x=\begin{pmatrix}
 1 & 1 & 0\\
 0 & 1 & 0\\
 0 & 0 & 1\\
\end{pmatrix}$, $y=\begin{pmatrix}
 1 & 0 & 0\\
 0 & 1 & 1\\
 0 & 0 & 1\\
\end{pmatrix}$ and $z=\begin{pmatrix}
 1 & 0 & 1\\
 0 & 1 & 0\\
 0 & 0 & 1\\
\end{pmatrix}$.

Those elements satisfy the relations $xz=zx$, $yz=zy$ and $z=xyx^{-1}y^{-1}=[x,y]$.  Thus one obtains a presentation of the Heisenberg group over the integers. Note the formula, for $a$, $b$, $c\in{\bf Z}$

\[
x^{a}z^{b}y^{c}=\begin{pmatrix}
 1 & a & b\\
 0 & 1 & c\\
 0 & 0 & 1\\
\end{pmatrix}.
\]

From that perspective, the group law is given by
$$
x^{a}z^{b}y^{c}x^{a'}z^{b'}y^{c'}=x^{a+a'}z^{b+b'+a'c}y^{c+c'}.
$$

The Heisenberg group $H_{\bf Z}$ over ${\bf Z}$ can be identified with ${\bf Z}^{3}$ with the previous group law. The abelianization of $H_{\bf Z}$ is freely generated by the images of $x$ and $y$ and is thus isomorphic to ${\bf Z}^{2}$. Thus $H_{\bf Z}$ is a central extension of ${\bf Z}^{2}$ by ${\bf Z}$.

Let $M$ and $N$ be natural integers. Let $L$ be a common divisor of $M$ and $N$. Let $H_{M,N,L}$ be the quotient group of $H_{\bf Z}$  spanned by $x$ and $y$ with relations $xz=zx$, $yz=zy$, $z=xyx^{-1}y^{-1}$ and $x^{N}=y^{M}=z^{L}=1$. Such a group can be identified (as a set) with ${\bf Z}/M{\bf Z}\times {\bf Z}/L{\bf Z}\times {\bf Z}/N{\bf Z}$ via the inverse of the map $(a,b,c)\mapsto x^{a}z^{b}y^{c}$. 

Note the map $H_{M,N,L}\rightarrow {\bf Z}/M{\bf Z}\times {\bf Z}/N{\bf Z}$ is coming from the abelianization.

Let $M'$, $N'$ and $L'$ be integer $\ge1$ such that $M|M'$, $N|N'$, and $L|L'$. The canonical group homomorphism $H_{M',N',L'}\rightarrow H_{M,N,L}$ is surjective. Its kernel is generated by $\{x^{M},y^{N}\}$. Since $[x^{M},y^{N}]=z^{NM}$, this kernel is abelian if and only if $L'|NM$.

\begin{prop}
Let $T$ be the lower common multiple of $M$ and $N$. 
The group $H_{M,N,L}$ is of exponent $T$ if $T$ is odd or $T/L$ is even. It is of exponent $2T$ otherwise.
In particular $H_{N,N,N}$, for $N$ odd, and $H_{N,N,N/2}$, for $N$ even, are of exponent $N$.
\end{prop}
\begin{proof}
Let $\alpha$, $\beta\in H_{M,N,L}$ in the subgroups generated by $x$ and $y$ respectively. 
Let $e=T$ if $T$ is odd or $T/L$ is even. Let $e=2T$ otherwise.
Since $[\alpha,\beta]$ belongs to the center of $H_{M,N,L}$, it is of order dividing $L$. 
Moreover, one has 
$(\alpha\beta)^{n}=\alpha^{n}[\alpha,\beta]^{n(n-1)/2}\beta^{n}$, which vanishes if $n=e$. By this calculation, one has $(xy)^{n}=x^{n}z^{n(n-1)/2}y^{n}$, which shows that $xy$ is of order $T$.
\end{proof}

Since the group ${\bar\Ga}(2)=\Gamma(2)/\{\pm1\}$ is freely generated by $A$ and $B$, one gets a surjective group homomorphism ${\bar\Ga}(2)\rightarrow H_{\bf Z}$ which sends $A$ and $B$ to $x$ and $y$ respectively. Its kernel is $[{\bar\Ga}(2),[{\bar\Ga}(2),{\bar\Ga}(2)]]$.
Every (necessarily nilpotent) finite quotient group of ${\bar\Ga}(2)$ which factorizes through ${\bar\Ga}(2)/[{\bar\Ga}(2),[{\bar\Ga}(2),{\bar\Ga}(2)]]$ is isomorphic to one of the groups $H_{M,N,L}$. 

Denote by  $\Ga_{M, N, L}$ the kernel of the map : 
${\bar\Ga}(2)\rightarrow H_{\bf Z}\rightarrow  H_{M,N,L}$. 

\subsection{The lower central series}
Recall that the {\it lower central series} $(G_k)_{k\ge1}$ of a group $G$ is defined recursively by $G_1=G$ and $G_{k+1}=[G_k,G]$. The quotient $G_k/G_{k+1}$ is then an abelian group. When $G$ is a free group generated by the family $(t_i)_{i\in I}$, $G_k/G_{k+1}$ is a free abelian group generated (not freely) by the classes of the commutators of weight $k$ on the generators, {\it i.e.} by the $[t_{i_1},...,t_{i_k}]$, where the indices run through any sequence $\{1,2,...,k\}\rightarrow I$ \cite[Theorem 10.2.3]{MR0103215}. When, furthermore, $I$ is finite of cardinality $m$, the rank $r_k$ of $G_k/G_{k+1}$ is given by Witt's formula, involving the necklace polynomial,
$$
r_k(G)=\frac{1}{k}\sum_{d|k}\mu(d)m^{k/d},
$$
where $\mu$ is the M\"obius function.
In particular the lower central series $({\bar\Ga}(2)_k)_{k\ge1}$  of ${\bar\Ga}(2)$ satisfies $r_1({\bar\Ga}(2))=2$, $r_2({\bar\Ga}(2))=1$, $r_3({\bar\Ga}(2))=2$ and $r_4({\bar\Ga}(2))=3$.
The corresponding generators of ${\bar\Ga}(2)_k/{\bar\Ga}(2)_{k+1}$ for $k=1$, $2$, $3$ are the classes of $\{A,B\}$, $\{C\}$, and $\{[C,A],[C,B]\}$ respectively. Therefore one has a surjective group morphism $\phi_1$ : ${\bar\Ga}(2)\rightarrow \Z^2$ such that $\phi_1(A)=(1,0)$ and $\phi_1(B)=(0,1)$. Its kernel is ${\bar\Ga}(2)_2$. Moreover, one gets a group isomorphism:
${\bar\Ga}(2)_1/{\bar\Ga}(2)_{3}\simeq H_{\bf Z}$.

Next we have the surjective group morphism $\phi_2$ : ${\bar\Ga}(2)_2\rightarrow \Z$, such that $\phi_2(C)=1$. Its kernel is ${\bar\Ga}(2)_3$. We can now describe $\phi_3$ : ${\bar\Ga}(2)_3\rightarrow \Z^2$ such that $\phi_3([C,A])=(1,0)$ and $\phi_3([C,B])=(0,1)$. Something interesting happens at that stage.

For $k$ integer $\ge 2$, the extension $1\rightarrow G_{k}/G_{k+1}\rightarrow G_{k-1}/G_{k+1}\rightarrow G_{k-1}/G_k\rightarrow 1$ is central. Consequently, since $r_2({\bar\Ga}(2))=1$, the group ${\bar\Ga}(2)_2/{\bar\Ga}(2)_4$ is abelian, and free of rank $3$. Of course, the extension $1\rightarrow {\bar\Ga}(2)_{2}/{\bar\Ga}(2)_{4}\rightarrow {\bar\Ga}(2)_{1}/{\bar\Ga}(2)_{4}\rightarrow {\bar\Ga}(2)_{1}/{\bar\Ga}(2)_2\rightarrow 1$ is not central.

\begin{prop}
There exists a group isomorphism $\psi$ : ${\bar\Ga}(2)_2/{\bar\Ga}(2)_4\simeq \Z^3$ given by $C\mapsto (0,0,1)$, $[C,A]\mapsto (1,0,0)$, and $[C,B]\mapsto(0,1,0)$. One has, for $\gamma\in {\bar\Ga}(2)$ and $\delta\in\bar\Ga(2)_2$, the formula 
$$
\psi(\gamma\delta\gamma^{-1})=(-\phi_1(\gamma)\phi_2(\delta),0)+\psi(\delta).
$$
or equivalently
$$
\psi([\delta,\gamma])=(\phi_1(\gamma)\phi_2(\delta),0).
$$
In particular, one has, for $i$, $j\in\Z$, the formula $\psi(A^iB^jC^kB^{-j}A^{-i})=(-ki,-kj,k)$.
\end{prop}

\begin{proof}
Given that $\{[C,A],[C,B]\}$ is basis of the $\Z$-module ${\bar\Ga}(2)_3/{\bar\Ga}(2)_4$, and $C$ is a basis of the $\Z$ module ${\bar\Ga}(2)_2/{\bar\Ga}(2)_3$, any lifting of $C$ modulo ${\bar\Ga}(2)_3$, to a class $C'$ modulo ${\bar\Ga}(2)_4$ gives a basis $\{[C,A],[C,B], C'\}$ of ${\bar\Ga}(2)_2/{\bar\Ga}(2)_4$. The choice $C=C'$ is evidently suitable, but is somewhat arbitrary.

One has 
$$
\psi(\gamma\delta\gamma^{-1})=\psi(\gamma\delta C^{-\phi_2(\delta)}\gamma^{-1}C^{\phi_2(\delta)}\delta^{-1}\delta C^{-\phi_2(\delta)}\gamma C^{\phi_2(\delta)}\gamma^{-1}).
$$
Since $\delta C^{-\phi_2(\delta)}\in\bar\Ga(2)_3$, the factor $\gamma\delta C^{-\phi_2(\delta)}\gamma^{-1}C^{\phi_2(\delta)}\delta^{-1}$ belongs to $\bar\Ga(2)_4$; hence its image by $\psi$ vanishes. So we get
\begin{equation}
\label{psi}
\psi(\gamma\delta\gamma^{-1})=\psi(\delta)-\phi_2(\delta)\psi(C)+\phi_2(\delta)\psi(\gamma C\gamma^{-1}),
\end{equation}
which translates into 
$$
\psi([\delta,\gamma])=\phi_2(\delta)\psi([C,\gamma]).
$$
It remains to determine $\psi([C,\gamma])$. Let $\gamma_1$, $\gamma_2\in {\bar\Ga}(2)$. One has
\begin{equation}
\label{doubleconjugacy}
(\gamma_1\gamma_2)\delta(\gamma_1\gamma_2)^{-1}=[\gamma_1,\gamma_2][\gamma_2,[\gamma_1,\delta]][\gamma_1,\delta][\gamma_2,\delta]\delta[\gamma_2,\gamma_1].
\end{equation}
Since $[\gamma_2,[\gamma_1,\delta]]$ is a commutator of degree $4$, it is in the kernel of $\psi$. It follows that the map $\gamma\mapsto \psi(\gamma\delta\gamma^{-1})$ is a group homomorphism from $\bar\Ga(2)$ to $\Z^2$. For $\delta=C$, $\gamma=A$ (resp. $\gamma=B$), one has $\psi(\gamma\delta\gamma^{-1})=-(\psi_1(\gamma),0)$, hence the latter equality is true for all $\gamma\in\bar\Ga(2)$. Thus one gets
$$
\psi([C,\gamma])=(\phi_1(\gamma),0),
$$
which gives the main formula. It remains to apply this to $\gamma=A^iB^j$ and $\delta=C^k$ to obtain the final formula.

\end{proof}

The exact sequence  $1\rightarrow {\bar\Ga}(2)_{3}/{\bar\Ga}(2)_{4}\rightarrow {\bar\Ga}(2)_{1}/{\bar\Ga}(2)_{4}\rightarrow {\bar\Ga}(2)_{1}/{\bar\Ga}(2)_3\rightarrow 1$ identifies to a central group extension
$$
0\rightarrow \Z^2\rightarrow H'_\Z\rightarrow H_\Z\rightarrow 1.
$$

\subsection{The groups $\Phi_N$, $\Phi'_N$ and $\Phi''_N$}
We denote the group $\Ga_{N,N,1}$ by $\Phi_N$. It is the kernel of the group homomorphism $\bar\phi_1$ : ${\bar\Ga}(2)\rightarrow (\Z/N\Z)^2$, which maps $A$ to $(1,0)$ and $B$ to $(0,1)$ (and thus $C$ to $(0,0)$). 
A system of representatives for the cosets $\Phi_N \backslash {\bar\Ga}(2)$ is given by $A^i B^j$ with $i,j \in \{0,1...(N-1)\}$.
The structure of $\Phi_N$ is probably well known, but we could not find a complete reference for a presentation $\Phi_N$.

\begin{prop}
The group $\Phi_N$ is generated by $U=\{A^N,B^N, A^iB^jCB^{-j}A^{-i}/0\le i,j\le N-1\}$. Moreover $\Phi_N$ has presentation $<A^N, B^N, T_{i,j}, 0\le i,j\le N-1 |[A^N,B^N]=\prod_{i=0}^{N-1}\prod_{j=0}^{N-1}T_{N-1-i,j}>$.
\end{prop}
\begin{proof}

The group $\Phi_N$ is generated by $U=\{A^N,B^N, A^iB^jCB^{-j}A^{-i}/0\le i,j\le N-1\}$ ~\cite{MR681120}, ~\cite{MR0441978}. By the Nielsen-Schreier theorem, $\Phi_N$, being a free subgroup of index $N$ of a free group on two generators, is free on 
$N^2+1$ generators.
A relation between the $N^2+2$ exhibited generators in $U$ presents itself : 
\begin{equation}
\label{relation}
A^NB^NA^{-N}B^{-N}=\prod_{i=0}^{N-1}\prod_{j=0}^{N-1}A^{N-1-i}B^jCB^{-j}A^{i+1-N}.
\end{equation}
One of the exhibited generators in $U$ can be expressed in terms of the $N^2+1$ remaining elements of $U$. Thus we get a presentation by generators and relations of $\Phi_N$.
\end{proof}

Set $N'=N$ if $N$ is odd, and $N'=N/2$ if $N$ is even.

We set $\Phi'_N=\Ga_{N,N,N'}$. It is the subgroup of $\Phi_N$ obtained as the kernel of the morphism $\bar\phi_2$ : $\Phi_N\rightarrow \Z/N'\Z$ which vanishes on $A^N$ and $B^N$ and takes the value $1$ on $A^iB^jCB^{-j}A^{-i}$ for $i$, $j$ integers.

Alternately, $\Phi'_N$ is the kernel of the composed map ${\bar\Ga}(2)\rightarrow H_{\bf Z}\rightarrow  H_{N,N,N'}$. Thus the composed map $[{\bar\Ga}(2),{\bar\Ga}(2)]\rightarrow \Z\rightarrow \Z/N'\Z$  extends to a map
$\bar\phi_2$: $\Phi_N\rightarrow \Z/N'\Z$, with kernel $\Phi'_N$. It vanishes on $A^N$ and $B^N$.

\begin{prop}
The composed map $[{\bar\Ga}(2),{\bar\Ga}(2)]\rightarrow \Z^3\rightarrow (\Z/N'\Z)^3$ extends to a group homomorphism
$$
\bar\psi : \Phi_N\rightarrow (\Z/N'\Z)^3
$$
which vanishes on $A^N$ and $B^N$ and, for $\gamma\in\bar\Ga(2)$ and $\delta\in\Phi_N$, satisfies 
\begin{equation}
\label{barpsi}
\bar\psi(\gamma\delta\gamma^{-1})=(-\bar\phi_1(\gamma)\bar\phi_2(\delta),0)+\bar\psi(\delta).
\end{equation}
In particular, for $i$, $j$, $k\in\Z$, one has 
\begin{equation}
\label{psigen}
\bar\psi(A^iB^jC^kB^{-j}A^{-i})=(-ik,-jk,k).
\end{equation}
\end{prop}
\begin{proof}
It follows from the presentation of $\Phi_N$ that formula \ref{psigen} defines a morphism $\Phi_N\rightarrow (\Z/N\Z)^3$ (it vanishes on the exhibited relation \ref{relation} between the generators). Such a morphism coincides with the composed map $[{\bar\Ga}(2),{\bar\Ga}(2)]\rightarrow \Z^3\rightarrow (\Z/N\Z)^3$, by the formula we established on $\psi$. 

The formula \ref{barpsi} is valid whenever $\delta\in\bar\Ga(2)_2$, by \ref{psi}. Since $\Phi_N$ is generated by $\bar\Ga(2)_2$ together with $A^N$ and $B^N$, it remains to prove \ref{barpsi} for $\delta=A^N$ and $\delta=B^N$. Consider the case $\delta=B^N$ for instance (the other case is similar).
Let us use the formula \ref{doubleconjugacy} again.
We get, for $\gamma_1$, $\gamma_2\in\bar\Ga(2)$, (leaving out opposite terms)
$$
\bar\psi(\gamma_1\gamma_2 B^N\gamma_2^{-1}\gamma_1^{-1})=\bar\psi([\gamma_2,[\gamma_1,B^N]])+\bar\psi([\gamma_1,B^N])+\bar\psi([\gamma_2,B^N])+\bar\psi(B^N).
$$
We have $\bar\psi(B^N)=0$, and $\bar\psi([\gamma_2,[\gamma_1,B^N]])$ is the reduction modulo $N$ of $\psi([\gamma_2,[\gamma_1,B^N]])$. One has $\psi([\gamma_2,[\gamma_1,B^N]])=-(\phi_1(\gamma_2)\phi_2([\gamma_1,B^N],0)$. 
But $\phi_2([\gamma_1,B^N])\in N\Z$. So we have proved that the map $\gamma\mapsto \bar\psi(\gamma B^N\gamma^{-1})$ is a group homomorphism. 
It remains to prove that the formula \ref{barpsi} holds for $\gamma=A$ and $\gamma=B$ (still in the configuration where $\delta=B^N$). Only the case $\gamma=A$ is of interest. We have
$$
\bar\psi(AB^NA^{-1})=\bar\psi(AB^NA^{-1}B^{-N}).
$$
As $AB^NA^{-1}B^{-N}\in\bar\Ga(2)_2$, $\bar\psi(AB^NA^{-1}B^{-N})$ is the reduction modulo $N$ of $\psi(AB^NA^{-1}B^{-N})$. We use the identity
$$
AB^NA^{-1}B^{-N}=(ABA^{-1})^NB^{-N}=(ABA^{-1}B^{-1}B)^NB^{-N}=CBCB^{-1}B^2CB^{-2}...B^{N-1}CB^{1-N}.
$$
The right-hand side is a product of generators of $\Phi_N$. We can apply \ref{psigen}
$$
\psi(AB^NA^{-1}B^{-N})=\sum_{i=0}^{N-1}\psi(B^iCB^{-i})=\sum_{i=0}^{N-1}(0,-i,0)=(0,N(N-1)/2,0).
$$
Since $N'={\rm gcd}(N,N(N-1)/2)$, we have indeed that $\bar\psi(AB^NA^{-1}B^{-N})=0$. 
\end{proof}

\begin{remark}
Let $N''=N'$ if $N$ is prime to $3$, and $N''=N'/3$ otherwise.
The appearance of the denominator $2$ (for $N'$) and now $3$ (for $N''$) is related to Bernoulli numbers. We suspect that ultimately it is related to the mixed Tate motives that have been discovered by Deligne in his study of the nilpotent completion of the fundamental group of the projective line deprived of three points \cite{MR1012168}.
\end{remark}

We define $\Phi''_N$ as the kernel of the composed maps $\Phi_{N}\rightarrow (\Z/N''\Z)^{2}\times(\Z/N'\Z)$. 

\begin{cor}
The exact sequence 
$$
0\rightarrow \Phi_N'/\Phi_N''\rightarrow {\bar\Ga}(2)/\Phi_N''\rightarrow {\bar\Ga}(2)/\Phi_N'\rightarrow 0
$$
makes of the group ${\bar\Ga}(2)/\Phi_N''$ a central extension of the Heisenberg group $H_{N,N,N'}$ by $(\Z/N''\Z)^2$.
\end{cor}
\begin{proof}
It follows from formula \ref{barpsi}, as $\bar\phi_2$ vanishes on $\Phi'_N$. 
\end{proof}
We denote by $H'_{\Z/N\Z}$ the group ${\bar\Ga}(2)/\Phi_N''$

\begin{prop}
The group $H'_{\Z/N\Z}$ is of exponent $N$. In other words, for every $\gamma\in\bar\Ga(2)$, one has $\gamma^N\in\Phi''_N$.
\end{prop}
\begin{proof}
It relies on relations for commutators, that, we presume, are well known. Let $G$ be a group. Suppose $G_4$ is trivial. Then $G_3$ is contained in the center of $G$. Let  $\alpha$, $\beta\in G$. Set $\gamma=[\alpha,\beta]$, $\alpha'=[\gamma^{-1},\alpha]$ and $\beta'=[\gamma^{-1},\beta]$. Let $n$ be an integer.
One has the relation 
\begin{equation}
\label{betanalpha}
\beta^n\alpha=\alpha\gamma^{-n}\beta^n\alpha'^n\beta'^{-n(n-1)/2}.
\end{equation}
We prove it by induction on $n$. Indeed, it holds for $n=0$. Suppose it holds for some value of $n$. We have 
$$
\beta^{n+1}\alpha=\beta\alpha\gamma^{-n}\beta^n\alpha'^n\beta'^{-n(n-1)/2}=\gamma^{-1}\alpha\beta\gamma^{-n}\beta^n\alpha'^n\beta'^{-n(n-1)/2}.
$$
We use the relations $\gamma^{-1}\beta=[\gamma^{-1},\beta]\beta\gamma^{-1}$ and $\gamma^{-1}\alpha=[\gamma^{-1},\alpha]\alpha\gamma^{-1}$.
Thus we get
$$
\beta^{n+1}\alpha=\alpha\gamma^{-n-1}\beta^{n+1}\alpha'^{n+1}\beta'^{-n-n(n-1)/2},
$$
which is the desired formula. We pass to the next step. We have
$$
(\alpha\beta)^n=\alpha^n\gamma^{-n(n-1)/2}\beta^n\alpha'^{-n(n-4)(n+1)}\beta'^{n(n-1)(n+1)/6}.
$$
We proceed again by induction on $n$. We suppose the formula holds for a certain value of $n$. We get
$$
(\alpha\beta)^{n+1}=\alpha^n\gamma^{-n(n-1)/2}\beta^n\alpha'^{-n(n-4)(n+1)}\beta'^{n(n-1)(n+1)/6}\alpha\beta.
$$
Using the formula \ref{betanalpha}, we get 
$$
(\alpha\beta)^{n+1}=\alpha^n\gamma^{-n(n-1)/2}\alpha\gamma^{-n}\beta^{n+1}\alpha'^{-n(n-4)(n+1)+n}\beta'^{n(n-1)(n+1)/6-n(n-1)/2}.
$$
We now use the formula  $\gamma^{-1}\alpha=[\gamma^{-1},\alpha]\alpha\gamma^{-1}$ and get
$$
(\alpha\beta)^{n+1}=\alpha^{n+1}\gamma^{-n(n-1)/2-n}\beta^{n+1}\alpha'^{-n(n-4)(n+1)+n-n(n-1)/2}\beta'^{n(n-1)(n+1)/6}, 
$$
which is the desired formula.

Consider the case where $G=H'_{\Z/N\Z}$. Since this group is generated by the classes of $A$ and $B$, which are of order $N$, it follows that all elements of $H'_{\Z/N\Z}$ are of order divisible by $N$. 
\end{proof}

\begin{remark}
Let $p$ be prime number. Stallings introduced the lower $p$-central series $(S_{k})_{k\ge1}$ as a particular case for $N=p$ of the following construction. One has $S_1=G$, and, for $k\ge2$, $S_{k+1}=[G,S_k](S_k)^N$, where the latter expression is the subgroup of $G$ generated by $[G,S_k]$ and $(S_k)^N$. 
Note that, when $G= {\bar\Ga}(2)$, one has $S_2=[{\bar\Ga}(2),{\bar\Ga}(2)]{\bar\Ga}(2)^N=\Phi_N$ and $S_3=[{\bar\Ga}(2),\Phi_N]\Phi_N^N$. Note that $S_3\subset \Phi_N'\subset S_2$. Since $A^N$ and $B^N$ do not belong to $S_3$, the groups $S_3$ and $\Phi'_N$ do not coincide.
\end{remark}

\subsection{Odd adelic completions}
Recall that, for $k\ge 1$, ${\bar\Ga}(2)_k$ is the $k$-th term in the lower central series of ${\bar\Ga}(2)$. Let $D_3=\{\pm {\rm Id}, \pm \begin{pmatrix}
 0 & -1 \\
 1 & 0 \\
\end{pmatrix}, \pm\begin{pmatrix}
 -1 & 1 \\
 1 & 1 \\
\end{pmatrix},  \pm\begin{pmatrix}
 1 & 1 \\
 1 & -1 \\
\end{pmatrix}\}$ in  ${\rm PSL}_{2}(\Z/3\Z)$
be the index $3$, $2$-Sylow subgroup of ${\rm PSL}_{2}(\Z/3\Z)$. It is isomorphic to the Klein group.

Recall that the derived subgroup of ${\rm PSL}_{2}(\Z)$ is the projective congruence subgroup of level $6$ whose image in ${\rm PSL}_{2}(\Z/3\Z)$ is $D_{3}$, and whose image in ${\rm PSL}_{2}(\Z/2\Z)$ is cyclic of order $3$.

Let 
\[
\hat\Z_{\rm odd}=\varprojlim_{n\,{\rm odd}} \Z/n\Z\simeq\prod_{p\ne2}\Z_{p}
\]
 be the profinite completion of $\Z$ away from the prime $2$.
Let $\hat D_{{\rm odd}}$ be the inverse image of $D_{3}$ in ${\rm PSL}_{2}(\hat\Z_{\rm odd})$. 

\begin{prop}
The image of ${\bar\Ga}(2)_2$ in ${\rm PSL}_{2}(\Z/3\Z)$ is equal to $D_{3}$. For $p$ prime, $p>3$, its image modulo $p$ is ${\rm PSL}_{2}(\Z/p\Z)$.
\end{prop}
\begin{proof} 
Indeed, the images of ${\bar\Ga}(2)$ and ${\rm PSL}_{2}(\Z))$ in ${\rm PSL}_{2}(\Z/3\Z)$ coincide by weak approximation. Thus ${\bar\Ga}(2)_2$ modulo $3$ coincides with the derived subgroup of ${\rm PSL}_{2}(\Z/3\Z)$, which in turn is the reduction modulo $3$ of the derived subgroup of ${\rm PSL}_{2}(\Z))$. The second assertion is proved similarly.
\end{proof}

\begin{prop}
\label{oddadelic}
Let $k$ be an integer $\ge2$. The closure of ${\bar\Ga}(2)_k$ in ${\rm PSL}_{2}(\hat\Z_{\rm odd})$ is equal to $\hat D_{{\rm odd}}$.
\end{prop}
\begin{proof}
We prove this first for $k=2$. Let $n$ be an odd integer divisible by $3$. Let $D_n$ be the inverse image of $D_{3}$ in ${\rm PSL}_{2}(\Z/n\Z)$. The image of ${\bar\Ga}(2)_2$ modulo $n$ coincide with the image of the derived subgroup of ${\rm PSL}_{2}(\Z))$ modulo $n$, which is a subgroup of index $3$ of ${\rm PSL}_{2}(\Z/n\Z)$. 
Such a subgroup can only be $D_n$. Thus we obtain the proposition for $k=2$.

We prove the proposition for $k=3$.
Let $I_n$ be the image of ${\bar\Ga}(2)_3$ in ${\rm PSL}_{2}(\Z/n\Z)$. 
Note that we have an exact sequence 
\[
1\rightarrow K_n \rightarrow I_n\rightarrow D_3\rightarrow 1.
\]
Since we have an exact sequence
\[
1\rightarrow \pm1+3{\rm M}_2( \Z/\frac{n}{3}\Z)_0 \rightarrow {\rm PSL}_{2}(\Z/n\Z)\rightarrow {\rm PSL}_{2}(\Z/3\Z)\rightarrow 1
\]
where ${\rm M}_2( \Z/\frac{n}{3}\Z)_0$ is the subgroup of ${\rm M}_2( \Z/\frac{n}{3}\Z)$ made of matrices of trace $0$, the equality $I_n=D_n$ would follow from the inclusion $1+3{\rm M}_2( \Z/\frac{n}{3}\Z)_0\subset K_n$.
It remains to establish the latter inclusion.
Since ${\rm PSL}_{2}(\Z/n\Z)$ is equal to its derived subgroup when $n$ is prime to $6$, by the Chinese remainder theorem, it is enough to prove it when $n$ is a power of $3$.

It follows from the equality $[(1+p{\rm M}_2(\Z_p))_{0},{\rm SL}_2(\Z_p)]=(1+p{\rm M}_2(\Z_p))_{0}$, valid for any prime $p$, and the fact that the closure of ${\bar\Ga}(2)_2$ in ${\rm PSL}_2(\Z_3)$ contains $[(1+3{\rm M}_2(\Z_3))_{0}$.

The general case $k\ge3$ of the proposition is now immediate. Indeed, since ${\bar\Ga}(2)_3$ and ${\bar\Ga}(2)_2$ have the same image modulo $n$, those images are the second and third respectively derived subgroups of ${\rm PSL}_{2}(\Z/n\Z)$. Thus the lower central series of ${\rm PSL}_{2}(\Z/n\Z)$ stabilizes to the image modulo $n$ of ${\bar\Ga}(2)_k$ for any $k\ge 3$.
\end{proof}

\begin{prop}
The closure of $\Phi_{N}$ in ${\rm PSL}_{2}(\hat\Z_{\rm odd})$ is ${\rm PSL}_{2}(\hat\Z_{\rm odd})$ if $3$ does not divide $N$. It is $\hat H_{{\rm odd}}$ if $3$ divides $N$.
\end{prop}
\begin{proof} 
If $3$ does not divide $N$, $\Phi_N$ contains a non-zero upper triangular matrix (for instance $A'^N$) which is not the identity modulo $3$. Its closure in  ${\rm PSL}_{2}(\hat\Z_{\rm odd})$ contains a maximal proper subgroup of index $3$, namely $\hat D_{{\rm odd}}$, and an element which is not in that subgroup. Therefore the closure is ${\rm PSL}_{2}(\hat\Z_{\rm odd})$.
If $3$ divides $N$, since both $A^{3N}$ and $B^{3N}$ vanish modulo $3$, the images of $\Phi_N$ and ${\bar\Ga}(2)_2$ coincide in ${\rm PSL}_{2}(\Z/3\Z)$. Hence the result. 

\end{proof}

\begin{prop}
\label{closureHeisenberg}
The closure of $\Phi'_{N}$ in ${\rm PSL}_{2}(\hat\Z_{\rm odd})$ is ${\rm PSL}_{2}(\hat\Z_{\rm odd})$ if $3$ does not divide $N$. It is $\hat H_{{\rm odd}}$ if $3$ divides $N$.
\end{prop}
\begin{proof} 
The closure of ${\bar\Ga}(2)_2$ and ${\bar\Ga}(2)_3$ coincide modulo $n$ for every $n$ divisible by $3$, as we have just established. Thus the closure of $\Phi'_N$ in ${\rm PSL}_{2}(\hat\Z_{\rm odd})$ contains ${\bar\Ga}(2)_2$. 
Since the group $\Phi'_{N}$ contains the matrices $A^N$ and $B^N$, its closure contains $\Phi_N$. Thus the closure of $\Phi'_N$ in ${\rm PSL}_{2}(\hat\Z_{\rm odd})$ is equal to the closure of $\Phi_{N}$ in ${\rm PSL}_{2}(\hat\Z_{\rm odd})$.

\end{proof}
Let $\bar\Gamma'(2)=\bar\Gamma(2)\cap \hat D_{{\rm odd}}$. It is a subgoup of index $3$ of $\bar\Gamma(2)$.
\begin{cor}
Let $\Gamma$ be a congruence subgroup of $\bar\Gamma(2)$. One has $\Gamma\Phi'_N= \bar\Gamma'(2)$ if $\Gamma\subset\bar\Gamma'(2)$ and $3$ divides $N$. One has $\Gamma\Phi'_N= \bar\Gamma(2)$ otherwise.

Let $n$ be an even integer. In particular, one has $\Gamma(n)\Phi'_{N}=\bar\Gamma(2)$ if $3$ does not divide either $n$ or $N$. 
One has  $\Gamma(n)\Phi'_{N}=\bar\Gamma'(2)$ if $3$ divides both $n$ and $N$.
\end{cor}

\section{The associated Riemann surfaces}

\subsection{The Riemann surface $X_{M, N,L}$}
Denote by $X_{M, N,L}$ the compactified modular curve defined by $\Ga_{M, N, L}$.
\begin{prop}
The genus $g_{M, N,L}$ of the curve $X_{M, N,L}$ is given by the following formulas. Denote by $T$ the lowest common multiple of $M$ and $N$.
Suppose $T$ is even and $T/L$ is odd, then one has
\[
g_{M, N,L}:=g(X_{M, N,L})=(NML-NL-ML-NML/2T)/2+1.
\]
Suppose $T$ is odd or $T/L$ is even, then one has
\[
g_{M, N,L}:=g(X_{M, N,L})=(NML-NL-ML-NML/T)/2+1.
\]

\end{prop}
\begin{proof}
We use Riemann-Hurwitz formula for the morphism $X_{N,M,L}\rightarrow X(2)$. Since $\Gamma(2)$ has no elliptic elements, the ramification points of this morphism reside entirely among the cusps.

Concerning the cusps above $0$ (resp. $\infty$), note that the stabilizer of the rational number $0$ (resp. $\infty$) in $P\Gamma(2)$ is generated by $B$ (resp. $A$). As the morphism  $X_{N,M,L}\rightarrow X(2)$ is Galois, the ramification index is independent of the chosen cusp, and is the order of the orbit of $B$ (resp. of $A$) acting on $\Gamma_{M,N,L}\backslash\Gamma(2)$. This is $N$ (resp. $M$) by definition of $\Gamma_{M,N,L}$. Concerning the cusps above $1$, the stabilizer in $P\Gamma(2)$ of the rational number $-1$ is generated by $A^{-1}B$. It remains to determine the order of $A^{-1}B$ in $\Gamma_{M,N,L}\backslash\Gamma(2)=H_{M,N,L}$.

The abelianization provides a map : $H_{M,N,L}\rightarrow {\bf Z}/M{\bf Z}\times {\bf Z}/N{\bf Z}$ which sends $A^{-1}B$ to $(1,-1)$. The latter element is of order $T$, which leads us to examine the order of $(A^{-1}B)^{T}$ in $H_{M,N,L}$.
Denote by $D=[A^{-1},B]$. It belongs to and generates the center of $H_{M,N,L}$. 
It is of order $L$. Note the formula $A^{-1}B=DBA^{-1}$. 
Thus one has $(A^{-1}B)^{T}=D^{k}A^{-T}B^{T}$ in $H_{M,N,L}$, where $k$ is the number of factors $A^{-1}$ to the right of a factor $B$ in the $(A^{-1}B)^{T}$ written as a product of $2T$ factors. 
One has $k=1+2+...+(T-1)=T(T-1)/2$. This is why the order of $(A^{-1}B)^{T}$ is $1$ if $L$ is odd or if $T/L$ is even. This order is $2$ otherwise. 

Thus the ramification index $e$ of any cusp above $1$ is equal to $T$ if $L$ is odd or if $T/L$ is even. It is $2T$ otherwise. We can now apply the Riemann-Hurwitz formula for the dominant morphism $X_{M,N,L}\rightarrow X(2)$ :

\[
2 g_{M, N,L}-2=-2d+\sum_{x \in X_{M, N, L}}(e_x-1). 
\]

We have $d=|H_{M,N,L}|=MNL$. One gets

\[
2 g_{M, N,L}-2=-2MNL+NL(M-1)+ML(N-1)+NML(1-1/e)
\]
and 
\[
g_{M, N,L}=NML-NL-ML-NML/e+1
\]
The lemma follows from the calculation of $e$.

\end{proof}

\subsection{The Riemann surfaces $X_N$, $X'_N$ and $X''_N$}
All three groups $\Phi_N$, $\Phi_N'$ and $\Phi_N''$ act on the upper half-plane $\tH$. We denote respectively by $X_N$, $X'_N$ and $X''_N$ the corresponding completed modular curves. 

\begin{prop}
Both morphisms $X'_N\rightarrow X_N$ and $X''_N\rightarrow X'_N$ (and therefore $X''_N\rightarrow X_N$ as well) are unramified. The covering $X'_{N}\rightarrow X_{N}$ is cyclic of degree $N'$.
The covering  $X''_{N}\rightarrow X'_{N}$ is Galois with group $(\Z/N''\Z)^{2}$. The covering $X''_{N}\rightarrow X_{N}$ is abelian with Galois group $(\Z/N''\Z)^{2}\times(\Z/N'\Z)$.
\end{prop}
\begin{proof}
The first statement needs only to be established for the morphism $X''_N\rightarrow X_N$. The ramification points can only reside at the cusps. To show that those cusps are unramified, we need only to show that their width in $X''_N$ is equal to their width, equal to $N$, in $X_N$. Since the covering $X''_N\rightarrow X_1$ is Galois, the width of a cusp of $X_N''$ depends only on its image in the set of cusps $\{0,1,\infty\}$ of $X_1$.
We just need to look at the order of $A$, $B$ and $A^{-1}B$ in ${\bar\Ga}(2)/\Phi''_N$. All three are of order $N$ in the latter quotient.

The other assertions follow immediately from the properties of the groups $\Phi_N$, $\Phi_N'$ and $\Phi_N''$.

\end{proof}
The Riemann surface $X_N$ is isomorphic to the Riemann surface obtained from the complex points of the Fermat curve. The covering $X_N'\rightarrow X_N$ is obtained from what we call {\it Heisenberg covering of the Fermat curve} by passing to the complex numbers.

We obtain the genus $g_N$ and $g_N'$ of the curves $X_N$ and $X_N'$ by our formulas for the genus of $X_{N,M,L}$. One has $g_N=g_{N,N,1}=(N-1)(N-2)/2$. Furthermore, if $N$ is odd, one has $g'_{N}=g_{N,N,N}=(N^{3}-3N^{2}+2)/2=(N-1)(N^2-2N-2)/2$. If $N$ is even, one gets $g'_{N}=g_{N,N,N'}=(2N^3-5N^2+4)/4=(N-2)(2N^2-N-2)/4$.

The genus of $g_N''$ of the curve $X_N''$ can be deduced from the Riemann-Hurwitz formula. Since we have a covering $X''_N\rightarrow X'_N$ of degree $N''^2$, one has 
$$
g''_N=N''^2g'_N-N^2+1.
$$

The curve $X'_{N}$ possesses $NN'$ cusps above each of the cusps $0$, $1$ and $\infty$ of $X(2)$. 

\subsection{Some cases of small genus}

\begin{prop}
The genus $g$ of the curve $X_{M, N,L}$ is equal to $0$ for the following values of $(N,M,L)$, and only for those values : $(N,1,1)$, $(1,M,1)$, $(2,2,1)$, $(2,2,2)$.
\end{prop}
\begin{proof}
The formula just established gives : $2-2g=-L(MN-N-M-MN/e)$. Thus $g=0$ implies that $L=1$ or $2$. 

If $g=0$ and $L=1$, then one has $MN-N-M-MN/e=-2$ and $e={\rm lcm}(M,N)$. Thus ${\rm gcd}(M,N)$ divides $2$. If  ${\rm gcd}(M,N)=1$, then one has $MN-N-M-1=-2$, and thus $(M-1)(N-1)=0$ that is $M=1$ or $N=1$.
If ${\rm gcd}(M,N)=2$, one has $MN-N-M-2=-2$ and thus $(M-1)(N-1)=1$; therefore one has $(M,N)=(2,2)$.

If $g=0$ and $L=2$, then one has $MN-N-M-MN/e=-1$. If $4$ divides neither $M$ nor $N$, one has $e=2{\rm lcm}(M,N)$. Then ${\rm gcd}(M,N)$ divides $2$, and is equal to $2$. Then one has $MN-N-M-1=-2$, as above. As $L=2$, the cases $M=1$ and $N=1$ are excluded; thus one has $(M,N)=(2,2)$.

\end{proof}
\begin{prop}
The genus $g$ of the curve $X_{M, N,L}$ is equal to $1$ for the following values of $(N,M,L)$, and only for those values (up to permutation of $N$ and $M$) : $(3,2,1)$, $(4,2,1)$, $(4,2,2)$, $(3,3,1)$, $(3,3,3)$.

The Jacobian varieties of those curves are elliptic curves endowed with automorphisms of order $3$, $4$, $4$, $3$ and $3$ respectively. Consequently the $j$-invariants of those curves are $0$, $1728$, $1728$, $0$ and $0$ respectively.
\end{prop}
\begin{proof}
Consider again the formula : $2-2g=-L(MN-N-M-MN/e)$. Thus $g=1$ amounts to $MN-N-M-MN/e=0$. One has $MN-N-M-MN/e=(N-2)(M-1)-2+M(1-N/e)$. We can suppose that $N>1$ and $M>1$ (otherwise $g=0$).

If $N=2$, one gets $M(1-2/e)=2$. If $L=1$, then $e={\rm lcm}(M,2)$. Thus $M-1=2$ or $M-2=2$. One has $(N,M,L)=(2,3,1)$ or $(N,M,L)=(2,4,1)$.

If $N=2$ and $L=2$, then  $e=2M$  or $e=M$. If $e=M$, then $M-2=2$. If $e=2M$, then $M-1=2$ (absurd since $L|M$). One has $(N,M,L)=(2,4,2)$ or $(N,M,L)=(2,4,1)$.

If $N=3$, then one has $M-3+M(1-3/e)=0$ and $e={\rm lcm}(M,3)$. One can suppose that $M>2$. Thus one has $M=3$ and $L=1$ or $L=3$. One has $(N,M,L)=(3,3,1)$ or $(N,M,L)=(3,3,3)$.

The case where $M=2$ or $M=3$ are treated similarly. If $N>3$ and $M>3$, one has $(N-2)(M-1)-2+M(1-N/e)>0$, which precludes $g=1$.

The automorphisms comme from the action of the image of $A$ in $H_{N,M,L}$ which stabilizes a cusp and therefore is an automorphism of an elliptic curve. 

\end{proof}

We can derive some information on the Manin-Drinfeld principle in the genus $1$ cases.

\begin{prop}
Divisors of the form $(CP)-(P)$, where $P$ is any point (in particular a cusp) of $X_{3,3,3}$ (resp. $X_{4,2,2}$) and $C$ acts via the map $\bar\Gamma(2)\rightarrow H_{N,M,L}$, are of order dividing $3$ (resp. $2$). 
\end{prop}
\begin{proof}

The canonical morphism $X_{3,3,3}\rightarrow X_{3,3,1}$ is of degree $3$. It gives rise to an isogeny of degree $3$ on the Jacobians by Albanese functioriality. Thus the kernel of this isogeny is of order $3$. Moreover any divisor of the form $(CP)-(P)$ is in the kernel of the isogeny.

\end{proof} 

\subsection{The curve $X'(2)$}
Recall that the group $\bar\Gamma'(2)$ is the subgroup of index $3$ of $\bar\Gamma(2)$ that is the inverse image of the $2$-Sylow subgroup of ${\rm PSL}_2(\Z/3\Z)$. Denote by $X'(2)=X_{\Gamma'(2)}$ the corresponding modular curve. 

\begin{prop}
The curve $X'(2)$ is of genus $1$ and its $j$-invariant is $0$.
\end{prop}
\begin{proof}
Consider the morphism of degree $d=3$ : $X'(2)\rightarrow X(2)$. Since none of the matrices $A$ (generator of the stabilizer of the cusp $\infty$),  $B$ (generator of the stabilizer of the cusp $0$), and $AB^{-1}$ (generator of the stabilizer of the cusp $1$) are not in $\bar\Gamma(2)$, the morphism is totally ramified at each of the three cusps of $X(2)$, and ramified only over those points. The Riemann-Hurwitz formula expresses the genus $g$ of $X'(2)$ as $(2g-2)=-2d+\sum_P(e_P-1)=6-6=0$ (where $P$ runs through points of ramification and $e_P$ designates the ramification index at that point), hence $g=1$.

The curve $X'(2)$ has an automorphism (the class of $A$ in $\bar\Gamma(2)/\bar\Gamma'(2)$) of order $3$ which leaves fixed the cusp $\infty$. Since it is of genus $1$, it is an elliptic curve of with an automorphism of order $3$. It has necessarily $j$-invariant $0$. 
\end{proof} 

\begin{remark}
Since $\Gamma_{3,3,3}=\Phi'_3$ is a subgroup of index $3$ of $\Gamma'(2)$. One has a morphism $X_{3,3,3}\rightarrow X'(2)$ of degree $3$. Both curves involved are of genus $1$, so we have an isogeny of degree $3$.
Note that $\bar\Gamma'(2)$ is a congruence subgroup of level $12$ and a subgroup of index $3$ of the derived subgroup $\tilde\Gamma$ of ${\rm PSL}_2(\Z)$. 
The latter subgroup defined a modular curve $X_{\tilde\Gamma}$ of genus $1$, which happens to have $j$-invariant $0$.
Thus we get an isogeny $X'(2)\rightarrow X_{\tilde\Gamma}$ of degree $3$.
\end{remark}

We now prove theorem \ref{theoremcomparison}.

\begin{proof}
The correspondence is obtained by composing pushing to $X_{\Gamma. \Phi'_d}$ and pulling back to $X_\Gamma$.
By Proposition \ref{closureHeisenberg}, one has $\Gamma. \Phi'_d=\bar\Gamma(2)$ except if $3$ divides $N$ and $\Gamma$ is contained in $\bar\Gamma'(2)$. 
In the latter case, $\theta_{N,\Gamma}$ factorizes through the jacobian of $X(2)$, which is $0$. 
Otherwise, namely if $3$ divides $N$ and $\Gamma$ is contained in $\bar\Gamma'(2)$, $\theta_{N,\Gamma}$ factorizes through the surjective map $J'_N\rightarrow J'(2)$. Since $J'(2)$ is the jacobian of a curve of genus $1$, and $j$-invariant $0$, it is an elliptic curve of $j$-invariant $0$.
Moreover the map $J'(2)\rightarrow J_\Gamma$ has finite kernel. The result follows.
\end{proof}

It is well known that the modular curve $X_0(27)$ has $j$-invariant $0$, and that the Fermat curve is a model for $X_0(27)$. We might expect a connection between $X_0(27)$ and $X'_3$. Let $\Gamma=\bar\Gamma(2)\cap \Gamma_0(27)$.  It is a congruence subgroup. But it is not contained in $\Gamma'(2)$. Thus, if we apply theorem \ref{theoremcomparison} to the group $\Gamma$, we obtain, counterintuitively,  the $0$-morphism $J'_3\rightarrow J_{\Gamma}$. {\it A fortiori}, if we push forward $J_\Gamma\rightarrow J_0(27)$ we obtain $0$.

\subsection{Mixed homology groups}

In \cite{MR3951413}, the homology group of $X_N$ relative to the whole set of cusps is thoroughly studied, by the method of Manin ~\cite{MR0314846}.
We found fruitful in \cite{BanerjeeMerel} to consider the following slightly different point of view:  for $\Ga$ a subgroup of finite index of $\bar\Ga(2)$, the corresponding modular curve $X_\Ga$ covers $X(2)$, which admits three cusps: $\Ga(2)0$, $\Ga(2)1$ and $\Ga(2)\infty$. Let $\partial_\Ga^+$ (resp. $\partial_\Ga^-$) be the set of cusps above $\Ga(2)0\cup\Ga(2)\infty$ (resp. $\Ga(2)1$)). 
It is thus possible to consider the mixed homology group $\HH_1(X_{\Ga}-\partial_{\Ga}^{-}, \partial_{\Ga}^{+}; {\Z})$ (and its dual $\HH_1(X_{\Ga}-\partial_{\Ga}^{+}, \partial_{\Ga}^{-}; {\Z})$).
One gets a group isomorphism 
\[
\xi_{\Ga}^{+} : \Z[\Ga\backslash\bar\Ga(2)]\rightarrow \HH_1(X_{\Ga}-\partial_{\Ga}^{-}, \partial_{\Ga}^{+}; {\Z})
\]
 which, for $g\in\bar \Gamma(2)$, associates to $\Gamma g$ the class $\xi_\Ga(g)$ in $\HH_1(X_{\Ga}-\partial_{\Ga}^{-}, \partial_{\Ga}^{+}; {\Z})$ of a path from $g0$ to $g\infty$ in the upper half-plane. 

Consider now the case where $\Ga=\Phi'_N$.
To simplify notations, set $\partial^{+}=\partial^+$ and $\partial^{-}=\partial^-$.
Recall that we have a group isomorphism $\Phi_N'\backslash \Gamma(2)\rightarrow H_{N,N,N'}$ which, for $(a,b,c)\in({\Z})^{3}$,  to $\Phi'_NA^aC^cB^b$ associates $x^az^cy^b$. We get thus a group isomorphism
\[
\Z[ H_{N,N,N'}]\simeq  \HH_1(X'_N-\partial^-,\partial^+;\Z).
\]
 For $(a,b,c)\in({\Z}/N{\Z})^{2}\times(\Z/N'\Z)$, set $i(a,b,c)=\xi^{+}_{\Phi'_N}(\Phi'_{N}A^{a}C^{c}B^{b})$. 

Thus, $H_{N,N,N'}$ acts on the curve $X'_{N}$, and transitively on the sets of cusps of $X'_{N}$ above $0$, $1$ and $\infty$ respectively. Thus the stabilizer of a cusp is cyclic of order $N$. 

The long exact sequence of relative homology yields: 
\[
0 \rightarrow \HH_1(X'_N-\partial^-; \Z) \rightarrow \HH_1(X'_N-\partial^-,\partial^+;\Z) \xrightarrow{\delta_N} \Z[\partial^+]^0 \rightarrow 0
\]
where $\delta_N$ is the boundary map.
Similarly we have a dual exact sequence: 
\[
0 \rightarrow \Z[\partial^-]^0\xrightarrow{\delta_N^*}\HH_1(X'_N-\partial^-,\partial^+;\Z)  \rightarrow \HH_1(X'_N,\partial^+;\Z)  \rightarrow 0
\]
where $\delta_N^*$ is the dual boundary map. It induces
\[
0 \rightarrow \Z[\partial^-]^0\xrightarrow{\delta_N^*}\HH_1(X'_N-\partial^-;\Z)  \rightarrow \HH_1(X'_N;\Z)  \rightarrow 0.
\]
Hence $\HH_1(X'_N;\Z) $ can be described as a subquotient of the group $\HH_1(X'_N-\partial^-,\partial^+;\Z)$. We will make this explicit by spelling out the maps $\delta_N$ and $\delta_N^*$. 

The sets of cusps of $X'_N$ lying above $\infty$, $0$ and $1$ are respectively $\Phi_N'\backslash \Ga(2)/A^{\Z}$, $\Phi_N' \backslash \Ga(2)/B^{\Z}$ and $\Phi_N' \backslash \Ga(2)/(AB^{-1})^{\Z} $. All three sets can be understood as follows.

\begin{prop}
\label{bijections}
We have three bijective maps as follows:
\[
\Phi_N'\backslash \Ga(2)/A^{\Z} \rightarrow (\Z/N\Z)\times (\Z/N'\Z)
\]
given by $x^{a} z^cy^{b}\mapsto (b,c+ab)$,
\[
\Phi_N' \backslash \Ga(2)/B^{\Z} \rightarrow (\Z/N\Z)\times (\Z/N'\Z)
\]
given by $x^a z^c y^{b}\mapsto (a,c)$, and
\[
\Phi_N' \backslash \Ga(2)/(AB^{-1})^{\Z} \rightarrow (\Z/N\Z)\times (\Z/N'\Z)
\]
given by $x^a z^c y^{b}\mapsto (a+b,c-b(b+1)/2)$.
\end{prop}
\begin{proof}
The first two identifications are straightforward. We establish the third one. 
Let $k$ be an integer. One finds by induction on $k$, $x^a z^c y^{b}(xy^{-1})^{k}=x^{a+k}z^{c-kb+k(k-1)/2}y^{b-k}$. 
Since $(a+k)+(b-k)=a+b$ and 
\[
c-kb+k(k-1)/2-(b-k)(b-k+1)/2=c-b(b+1)/2,
\] 
the map passes indeed to the quotient $\Phi_N' \backslash \Ga(2)/(AB^{-1})^{\Z}$.
It is surjective (take $b=0$). Since there are $NN'$ cusps above $1$, it is bijective. 

\end{proof}

Denote by $j_\infty(b,c)$ the cusp $\Phi_N'A^{a}C^{c-ab}B^bA^{N\Z}$, for any $a\in \Z$, by $j_0(a,c)$ the cusp $\Phi_N'A^{a}C^{c}B^{b}B^{N\Z}$, for any $b\in \Z$, and $j_{1}(d,c)$ the cusp 
$\Phi_N'A^{a}C^{c-b(b+1)/2}B^{b}(AB^{-1})^{N\Z}$, for any $a$, $b\in \Z$ such that $a+b=d$. With these conventions we can express $\delta_{N}$.

\begin{prop}
\label{boundary}
Let $a$, $b$, $c\in\Z$. 
One has $\delta_N(i(a,b,c))=j_\infty(b,c-ab)-j_0(a,c)$.
\end{prop}
\begin{proof}
The boundary of the modular symbol $\{A^{a}C^{c}B^{b}0, A^{a}C^{c}B^{b}\infty\}$ is $[\Phi'_{N}A^{a}C^{c}B^{b}A^{\Z}]-[\Phi'_{N}A^{a}C^{c}B^{b}B^{\Z}]$, which translates immediately into the claimed formula.
\end{proof}

We use \cite[Proposition 5]{BanerjeeMerel} to determine $\delta_N^*$. 

\begin{prop}
\label{boundarydual}
One has, for $d\in\Z/N\Z$ and $c\in\Z/N'\Z$,
\[
\delta_N^*(j_{1}(d,c))=\sum_{a,b\in\Z/N\Z, a+b=d}i(a,b+1,c-b(b+1)/2)-i(a,b,c-b(b+1)/2).
\]
\end{prop}
\begin{proof}
We just need to translate the third statement  \cite[Proposition 5]{BanerjeeMerel}. With the notations of that proposition, we have $w_{1}=N$. 
It remains to use the third bijection of proposition \ref{bijections} and the definition of $i$.

\end{proof}
Let $S_{N}$ be the subgroup of $\Z[(\Z/N\Z)^{2}\times\Z/N'\Z]$ formed by the elements of the form 
$\sum_{a,b,c}\lambda_{a,b,c}[a,b,c]$, satisfying, for every $(a,c)\in(\Z/N\Z)\times (\Z/N'\Z)$, the relation
$\sum_{b\in\Z/N\Z}\lambda_{a,b,c}=0$
and, for every $(b,c)\in(\Z/N\Z)\times (\Z/N'\Z)$, the relation $\sum_{a\in\Z/N\Z}\lambda_{a,b,c+ab}=0$.
By Proposition \ref{boundary}, its image by $i$ has boundary $0$. 

Let $R_{N}$ be the subgroup of $\Z[(\Z/N\Z)^{2}\times\Z/N'\Z]$ spanned by elements of the form 
\[
e_{c,d}=\sum_{a,b\in\Z/N\Z, a+b=d}[a,b+1,c-b(b+1)/2)]-[a,b,c-b(b+1)/2],
\]
for $(d,c)\in(\Z/N\Z)\times (\Z/N'\Z)$. By Proposition \ref{boundarydual}, it is a subgroup of $S_{N}$.
Thus we get a presentation  by generators and relations of the homology of $X'_{N}$.

\begin{cor}
The map $i$ produces an exact sequence
\[
0\rightarrow R_{N}\rightarrow S_{N}\rightarrow\HH_1(X'_N;\Z) \rightarrow 0.
\]

\end{cor}

\begin{remark}
By exchanging the roles of $\partial^{+}$ and $\partial^{-}$, it is possible to give a dual presentation of $\HH_1(X'_N;\Z) $. We leave this to the reader.
\end{remark}

\
\section{The Heisenberg covering and its models}
\label{Fermat}
In this section, we assume $N$ to be odd. Therefore $N'=N$.
\subsection{Modular functions}
Let $z\in\tH$. For $q_2=\exp{(\pi i z)}$, consider the classical  $\lambda$-function \cite{MR983619}:
\[
\lambda(z)=16 q_2 \prod_{n \geq 1} \left( \frac{1+q_2^{2n}}{1+q_2^{2n-1}}\right)^8,
1-\lambda(z)=\prod_{n \geq 1} \left( \frac{1-q_2^{2n-1}}{1+q_2^{2n-1}}\right)^8. 
\]
From the above expression, it is clear that $\lambda(1)=1$ and $(1-\lambda(1))=0$. The $N$-th roots
\[
x:=\sqrt[N]{\lambda}, y:=\sqrt[N]{1-\lambda}
\]
define modular units for $\Phi_{N}$. We recover thus the familiar model of the Fermat curve.

Since the $\lambda$ function identifies $\sP^1-\{0,1,\infty\}$ to $Y(2)$,
a covering of $\sP^1-\{0,1,\infty\}$ can be understood as a covering of $Y(2)$, {\it i.e.} a modular curve.

\subsection{Reminder on Fermat curves}
The  $N$-th Fermat curve $F_{N}$ is given by the projective model:
\[
X^N+Y^N=Z^N.
\]
Fermat curves and their points at infinity (cusps) are studied extensively by Rohrlich \cite{MR2626317}, \cite{MR0441978}, V\'elu \cite{MR582434} and  Posingies~\cite{Posingies}. 
In particular, these authors consider the map
\[
\beta_N:F_N \rightarrow \sP^1
\]
given by $(X:Y:Z) \rightarrow (X^N:Z^N)$. The map $\beta_N$ is of degree $N^{2}$. It is ramified only above the points $0,1,\infty$. 
The corresponding ramification points are given by $a_j=(0:\zeta^j:1)$, $b_j=(\zeta^j:0:1)$, $c_j=(\epsilon\zeta^j:1:0)$, for $j\in\Z/N\Z$.

Recall that $\zeta$ is a primitive $N$-th root of unity and $\epsilon$ is a square root of $\zeta$. Each of the 
above points has ramification index $N$ over $\sP^1$. For all $j\in\Z/N\Z$, the cusps $a_j$, $b_j$, $c_j$ are all defined over the cyclotomic field $\Q(\mu_N)$. Among them, only $a_0$, $b_0$ and $c_0$ are defined over $\Q$.

\subsection{The cuspidal subgroup of the Fermat curve}

The divisors of following modular functions 
are given by:
\[
div(x-\zeta^j)=N b_j-\sum_j c_j, 
div(y-\zeta^j)=N a_j-\sum_j c_j, 
div(x-\epsilon\xi^j y)=N c_j-\sum_j c_j. 
\]

Rohrlich  \cite{MR0441978} has determined the structure of the cuspidal group of the Jacobian of $F_{N}$ (see also V\'elu's alternative proof and description \cite{MR582434} ). 
Since every cuspidal divisor on $F_{N}$ is annihilated by $N$ in the Jacobian, the cuspidal group is a quotient of $\Z/N\Z[\partial_{\Phi_{N}}]^{0}$. The additional relations are given as follows. Recall that $N$ is odd. 

\begin{thm}[Rohrlich ~\cite{MR0441978}]
The group $\mathcal{P}$ of principal divisors is spanned by the following set 
$$
\{\sum_{i=0}^{N-1}[a_{i}]-[P],\sum_{i=0}^{N-1}[b_{i}]-[P],\sum_{i=0}^{N-1}[c_{i}]-[P],
$$
$$
\sum_{i=0}^{N-1}i([a_{i}]-[b_{i}]),\sum_{i=0}^{N-1}i([a_{i}]-[c_{i}]),\sum_{i=0}^{N-1}i^{2}([a_{i}]+[b_{i}]+[c_{i}]-3[P])\},
$$
where $P$ is any cusp of $F_{N}$. Thus the cuspidal subgroup of the Jacobian of $F_N$ is a free $\Z/N\Z$-module of rank $3N-7$. 

\end{thm}

We set 
\[
D_{A}=\sum_{i\in\Z/N\Z}\{i\}[a_{i}],
\]
(resp. $D_{B}=\sum_{i\in\Z/N\Z}\{i\}[b_{i}]$, resp. $D_{C}=\sum_{i\in\Z/N\Z}\{i\}[c_{i}]$)
and
\[
f_{A}=\prod_{i\in\Z/N\Z}(-y+\zeta^{i})^{\{i\}}.
\]
\begin{cor}
The class of $D_A$ is of order $N$ in the cuspidal subgroup of the Jacobian of $F_N$. Moreover it is congruent to $D_B$ and $D_C$ modulo $\mathcal{P}$ .
\end{cor}
\begin{proof}
Indeed the divisor of $f_A$ is $ND_A$. By Rohrlich's theorem, $D_A$ is of order $N$. The congruence of $D_A$, $D_B$ and $D_C$ modulo $\mathcal{P}$ follows from Rohrlich's relations.
\end{proof}

\subsection{The covering as a function field extension}
By Rohrlich's theorem, the order of $D_{A}$ in the cuspidal group is $N$, therefore an $N$-th root of $f_{A}$ defines a cyclic covering $G\rightarrow F_{N}$ (a provisional notation, since $G$ will be shown to be equal to $F_N'$) of degree $N$. Such a morphism is indeed unramified, since the divisor of $f_{A}$  belongs to $N\Z[\partial_{\Phi_{N}}]^{0}$. Since $D_{A}-D_{B}$ is a principal divisor, exchanging the roles of $X$ and $Y$ would give the same covering. A similar reasoning with respect to $D_{C}$ holds. 
The cyclic covering $G\rightarrow F_{N}$ translates into a covering of Riemann surfaces $W\rightarrow X_{N}$.

Consider now the third term $[\Ga(2),[\Ga(2),\Ga(2)]]$ in the lower central series of $\Ga(2)$. It is not a subgroup of finite index of $\Ga(2)$, but one can still consider the Riemann surface $X_{[\Ga(2),[\Ga(2),\Ga(2)]]}$.

\begin{prop}
The covering $W\rightarrow X_{N}$ factorizes through $X_{[\Ga(2),[\Ga(2),\Ga(2)]]}$.
\end{prop}
\begin{proof}
Let $U\in \Ga(2)$ and $V\in [\Ga(2),\Ga(2)]]$. Let $g$ be an $N$-th root of $f_{A}$. One has to prove that $g$ is invariant under $[U,V]$, that is $g_{|UV}=g_{|VU}$. Note first that $y_{|U}=\zeta^{r}y$, with $r\in {\Z}$. One has 
$$
g^{N}_{|U}=\prod_{i\in\Z/N\Z}(-\zeta^{r}y+\zeta^{i})^{\{i\}}=\prod_{i\in\Z/N\Z}(-y+\zeta^{i-r})^{\{i\}}=\prod_{i\in\Z/N\Z}(-y+\zeta^{i})^{\{i+r\}}.
$$
Write $\{i+r\}=\{i\}+r+t$, with $t\in N\Z$. Thus we get
$$
g^{N}_{|U}=g^N\prod_{i\in\Z/N\Z}(-y+\zeta^{i})^{r}\prod_{i\in\Z/N\Z}(-y+\zeta^{i})^{t}.
$$
The last factor is obviously an $N$-th power. 
By the description of Rohrlich of the cuspidal subgroup of $F_{N}$, the factor $\prod_{i\in\Z/N\Z}(-y+\zeta^{i})^{r}$ is an $N$-th power in the function field of $F_{N}$. So there exists a function $h$ on $F_{N}$, and an integer $s$ such that $g_{|U}=g\zeta^{s}h$. Note that $h_{|V}=h$ and $g_{|V}=\zeta^{q}g$, with $q\in\Z$. Therefore, one has:
$$
g_{|UV}=g_{|V}\zeta^{s}h_{|V}=\zeta^{s+q}gh=\zeta^{q}g_{|U}=g_{|VU}.
$$

\end{proof}

\begin{prop}
Any covering of degree $N$ of the Fermat modular curve $X_{N}$ that factors through $X_{[\Ga(2),[\Ga(2),\Ga(2)]]}$, factors through the Heisenberg covering. 
\end{prop}
\begin{proof}
Such a covering corresponds to a cyclic quotient of order $N$ of $\Phi_{N}$. Denote by $\Gamma$ the corresponding cocyclic subgroup of $\Phi_{N}$. Since the covering is unramified at the cusps, $\Gamma$ contains the matrices $A^{N}$ and $B^{N}$. Recall that $\Phi_{N}$ is generated by $A^{N}$, $B^{N}$ and the commutator subgroup of $\Ga(2)$. Since the covering factors through $X_{[\Ga(2),[\Ga(2),\Ga(2)]]}$, the group $\Gamma$ contains $[\Ga(2),[\Ga(2),\Ga(2)]]$. 
In particular, $\Gamma$ contains $[A,C]$ and $[B,C]$. Thus the image of $C$ in $\Phi_{N}/\Gamma$ is of order $N$. Consequently, $\Gamma$ is the group generated by $[\Ga(2),[\Ga(2),\Ga(2)]]$, $A^{N}$, $B^{N}$, $C^{N}$.

\end{proof}

\begin{cor}
Let $\zeta$ be a primitive $N$-th root of unity in $\Q(\mu_N)$. The function 
\[
f_{A}'=\prod_{i\in\Z/N\Z}(-y+\zeta^{i})^{\{i\}}
\]
admits an $N$-th root in the function field of $F'_N$.

\end{cor}
\begin{proof}
Indeed, the covering of $F_{N}$ defined by an $N$-th root of $f_{A}'$ is cyclic of order $N$ and factorizes through $X_{[\Ga(2),[\Ga(2),\Ga(2)]]}$. Consequently, it is $F'_{N}$.
\end{proof}

Let $K$ be the field of fractions of the curve given in the introduction over the complex numbers. In inhomogeneous form, it is generated by the variable $X$, $Y$, $T_{\zeta}$ for every primitive $N$-th root of unity $\zeta$,  with the relations
$$
X^{N}+Y^{N}=1
$$ 
and, for every primitive $N$-th root of unity $\zeta$ (in ${\bf Q}(\mu_{N})$)
$$
\prod_{j=1}^{(N-1)/2}(Y-\zeta^{-j})^{j}T_{\zeta}^{N}=\prod_{j=1}^{(N-1)/2}(Y-\zeta^{j})^{j}.
$$

\begin{cor}
The function field of $F'_{N}$ over $\C$ is isomorphic to  $K$.
\end{cor}
\begin{proof}
Indeed, the function field of $F'_{N}$ is the subfield of $K$ generated by an $N$-th root of $f_{A}$ over the function field of $F_{N}$. Let $\zeta'$ be a primitive $N$-th root of unity. By the preceding corollary, $T_{\zeta'}$ belongs to $K$. Thus all of $K$ is contained in the function field of $F'_{N}$.
\end{proof}
Thus we have shown that the curve $F'_N$ extends the Riemann surface $X'_N$.
\begin{cor}
The Riemann surfaces $X'_N$ and $F'_N({\C})$ are isomorphic.
\end{cor}

\subsection{The Heisenberg covering}
Denote by $G$ the Galois group of the field extension $K|\Q(X^N)$.
It sits in an exact sequence
$$
1\rightarrow H_N\rightarrow G\rightarrow (\Z/N\Z)^\times\rightarrow 1
$$
by the transitive action of $G$ on $N$-th roots of unity, and the fact that the Heisenberg covering is defined over  $\Q(\zeta)$. Recall that for $i\in (\Z/N\Z)$, $\{i\}$ denotes the representative of $i$ in $\{-(N-1)/2,...., (N-1)/2\}$.

\begin{prop}
For $\sigma\in G$,  there exists $u$, $v$, $s\in (\Z/N\Z)$ and $r\in (\Z/N\Z)^\times$ such that $\sigma(X)=\zeta^uX$, $\sigma(Y)=\zeta^v(Y)$, $\sigma(\zeta)=\zeta^r$. For $\rho$ a representative of $r$ in $\Z$, 
$$
\sigma(T^\rho)=\zeta^sTX^{v}\prod_{i\in  (\Z/N\Z)}(-Y+\zeta^i)^{(\rho\{i/\rho\}-\{i\})/N}.
$$
Furthermore, $(u,v,s,r)\in (\Z/N\Z)^{3}\times (\Z/N\Z)^\times$ caracterizes $\sigma$.
\end{prop}
\begin{proof}
The first three identities are evident. A simple calculation establishes the last one. Indeed
$$
\sigma(T^\rho)^N=\sigma(T^N)^\rho=\sigma(\prod_{i\in  (\Z/N\Z)}(-Y+\zeta^i)^{\{i\}})^\rho=\prod_{i\in  (\Z/N\Z)}(-\zeta^vY+\zeta^{\rho i})^{\{i\}\rho}.
$$
By factoring $T^N$ and $\prod_i \zeta^{v\{i\}}$, and replacing the variable $i$ by $j=\rho i-v$, one gets
$$
\sigma(T^\rho)^N=T^N\prod_i\zeta^{v\{i\}}\prod_{j\in  (\Z/N\Z)}(-Y+\zeta^{j})^{\rho\{(j+v)/\rho\}}\prod_{j\in  (\Z/N\Z)}(-Y+\zeta^j)^{-\{j\}}.
$$
We use $\prod_i \zeta^{v\{i\}}=1$ and we get 
$$
\sigma(T^\rho)^N=T^N\prod_{j\in  (\Z/N\Z)}(-Y+\zeta^{j})^{\rho\{(j+v)/\rho\}-\{j\}-v}\prod_{j\in  (\Z/N\Z)}(-Y+\zeta^j)^{v}.
$$
Using the identity $X^N=(1-Y^N)=\prod_{j\in  (\Z/N\Z)}(-Y+\zeta^{j})$, we get 
$$
\sigma(T^\rho)^N=T^NX^{Nv}\prod_{j\in  (\Z/N\Z)}(-Y+\zeta^{j})^{\rho\{(j+v)/\rho\}-\{j\}-v}.
$$
The desired formula follows by taking $N$-th roots.
\end{proof} 
With the notations of the proposition, we note $\sigma=\sigma_{u,v,r,s}$. Murty and Ramakrishnan note that Deligne showed that $F'_N$ can be defined over $\Q$, without giving a reference. We spell out this assertion.

\begin{prop}
The curve $F'_N$ can be defined over $\Q$. More precisely, the surjective map $G\rightarrow (\Z/N\Z)^\times$ admits the map $r\mapsto \sigma_{0,0,r,0}$ as a section. Denote by $S$ the corresponding subgroup of $G$. It acts trivially on $N$-th roots of unity. Consequently, the field of invariants by $S$ in $K$ is the function field of a curve over $\Q$, and defines $F'_N$ over $\Q$. The Heisenberg covering $F'_N\rightarrow F_N$ extends also over $\Q$.
\end{prop}
\begin{proof}
Group theoretic arguments about the structure of $G$ as an extension imply easily the existence of the section. We will show that our explicit map is indeed a section. Let $r$, $r'\in (\Z/N\Z)^\times$ and $\rho$ and $\rho'$ be representatives of $r$ and $r'$ respectively in $\Z$. We need to check that $\sigma_{0,0,r',0}\sigma_{0,0,r,0}=\sigma_{0,0,rr',0}$, which need to be verified only by application on $T$, or equivalently on $T^{\rho\rho'}$, and on $\zeta$. This is trivial for $\zeta$. Here is the computation on $T^{\rho\rho'}$. We first simplify the formula of the previous proposition
$$
\sigma_{0,0,r,0}(T^{\rho\rho'})=T^{\rho'}\prod_{j\in  (\Z/N\Z)}(-Y+\zeta^j)^{(\rho\{j/\rho\}-\{j\})\rho'/N}.
$$
Thus one has
$$
\sigma_{0,0,r',0}\sigma_{0,0,r,0}(T^{\rho\rho'})=T\prod_{j\in  (\Z/N\Z)}(-Y+\zeta^j)^{(\rho'\{j/\rho'\}-\{j\})/N}\prod_{j\in  (\Z/N\Z)}(-Y+\zeta^{\rho'j})^{(\rho\{j/\rho\}-\{j\})\rho'/N}).
$$
A change of variable in the second product of the right-hand side gives
\[
\sigma_{0,0,r',0}\sigma_{0,0,r,0}(T^{\rho\rho'})=T\prod_{j\in  (\Z/N\Z)}(-Y+\zeta^j)^{(\rho\rho'\{j/\rho\rho'\}-\{j\})/N}=\sigma_{0,0,rr',0}(T^{\rho\rho'}).
\]

\end{proof}
\subsection{Automorphisms}
The group $(\Z/N\Z)^{2}\simeq\Ga(2)/\Phi_{N}$ acts on $F_{N}$ : $(i,j)\in(\Z/N\Z)^{2}$ acts by the rule $(x,y)\mapsto(\zeta^{i}x,\zeta^{j}y)$. Such an action is defined on $\Q(\mu_{N})$. It lifts to an action of  $H_{\Z/N\Z}\simeq\Ga(2)/\Phi_{N}'$ on  $F_{N}'$. 

\begin{lemma}
The action of $H_{\Z/N\Z}$ on $F_{N}'$ is defined over $\Q(\mu_{N})$. 
\end{lemma}
\begin{proof}
For $U\in H_{\Z/N\Z}$, one needs to check that the action of $U$ on a $\Q(\zeta)$-rational function is still  $\Q(\zeta)$-rational. It suffices to verify this for $g$, an $N$-th root of $f_{A}$. We repeat a previous calculation and get 
$$
g^{N}_{|U}=g^N\prod_{i\in\Z/N\Z}(-y+\zeta^{i})^{r}\prod_{i\in\Z/N\Z}(-y+\zeta^{i})^{t},
$$
with $t\in N\Z$. Both factors $\prod_{i\in\Z/N\Z}(-y+\zeta^{i})^{r}$ and $\prod_{i\in\Z/N\Z}(-y+\zeta^{i})^{t}$ are $N$-th power, (of, say, $g_{1}$ and $g_{2}$) in the function field of $F_{N}$ over $\Q(\zeta)$.
We thus find that $g_{|U}$ is equal to $\zeta^{p}gg_{1}g_{2}$, which belongs to the function field of $X_{N}$ over $\Q(\zeta)$. 

\end{proof}

\begin{remark}
We do not give an explicit algebraic model of the curve $X''_{N}$. But it can be obtained by taking $N''$-th roots of the functions whose divisors are $\sum_{i=0}^{N-1}i^{2}([a_{i}]-[P])\}$,  $\sum_{i=0}^{N-1}i^{2}([b_{i}]-[P])\}$ and  $\sum_{i=0}^{N-1}i^{2}([c_{i}]-[P])\}$. 

\end{remark}
\subsection{Regular integral models of Heisenberg curves}
We relate the Heisenberg covering to the model given in the introduction. 
Note that the inhomogeneous version of the projective model is given by:
$$
X^{N}+Y^{N}=1
$$ 
and, for every primitive $N$-th root of unity $\zeta$ (in a fixed cyclotomic extension of ${\bf Q}$)
$$
\prod_{j=1}^{(N-1)/2}(Y-\zeta^{-j})^{j}T_{\zeta}^{N}=\prod_{j=1}^{(N-1)/2}(Y-\zeta^{j})^{j}.
$$
We now prove Theorem~\ref{Modelthm} (in the spirit of {\it e.g.}  \cite[Proposition 1.1.13]{Curilla}).
\begin{proof}
We have indeed a scheme of relative dimension $1$ over ${\rm Spec}({\Z}[\mu_{N},1/N])$. We just have to establish the smoothness.
We use the Jacobian criterion ({\it e.g.}  \cite[p. 130, Theorem 2.19]{MR1917232}).

Since the Heisenberg covering is obtained by taking the $N$-th root of a function which does not vanish outside the zero locus of $XY$, all points away from $X=0$ and $Y=0$ are regular in all characteristics prime to $N$.
It remains to establish the regularity of a point $P_{\zeta}$ of the form $(X,Y)=(0,\zeta)$ or $(X,Y)=(\zeta,0)$ with $\zeta$ a primitive $N$-th root of unity. Suppose  $(X,Y)=(0,\zeta)$. For points of this form, one has $T_{\zeta}=0$.
We set
\[
F_{\zeta}=\prod_{j=1}^{(N-1)/2}(Y-\zeta^{-j})^{j}T_{\zeta}^{N}-\prod_{j=1}^{(N-1)/2}(Y-\zeta^{j})^{j}
\]
and compute the partial derivative at $P_{\zeta}$. One obtains
\[
\frac{\partial
F_{\zeta}}{\partial Y}(P_{\zeta})=-
\prod_{j=2}^{(N-1)/2}(\zeta-\zeta^{j})^{j},
\]
which is a cyclotomic unit that belongs to ${\bf Z}[\mu_{N},1/N]^{\times}$. Thus the jacobian matrix is non-zero over any fiber in characteristic prime to $N$. This is sufficient to establish the regularity of $P_{\zeta}$ over ${\rm Spec}({\bf Z}[\mu_{N},1/N])$ since the structural morphism is of relative dimension $1$.
This reasoning applies to the zero locus of $Y$, by exchanging the roles of $X$ and $Y$.
\end{proof}

\subsection{Cusps}
The term cusp refers to the cusps of the modular curve $X_\Ga$ attached to a subgroup $\Ga$ of $\Ga(2)$. These points are not intrinsic to the corresponding algebraic curves.  In the cases of interest to us, we have a morphism $X_N\rightarrow X(2)$, given by the function $x^N$. Thus the cusps of $F_N$ are the $3N$ points above the points $0$, $1$, and $\infty$ described above.

Since the morphism $F'_N\rightarrow F_N$ is unramified, the modular curve $F'_N$ possesses $3N^2$ cusps. 

\begin{prop}
If $N$ is prime to $3$, the cusps of $F'_N$ are defined over $\Q(\zeta)$. If $3$ divides $N$, the cusps of $F'_N$ are defined over the cyclotomic field generated by the $3N$-th roots of unity.
\end{prop}
\begin{proof}
Let $a$ be a cusp of $F'_{N}$ above $a_{0}$. It is defined over the field generated by $g(a)$ and $\Q(\zeta)$, where $g$ is an $N$-th root of $f_{A}$.
One has 
$$
f_{A}(a_{0})=\prod_{i\in\Z/N\Z}(-1+\zeta^{i})^{\{i\}}=\prod_{i=1}^{(N-1)/2}\frac{(-1+\zeta^{i})^{i}}{(-1+\zeta^{-i})^{i}}=\prod_{i=1}^{(N-1)/2}(-\zeta^{i})^{i}.
$$
Thus we get
$$
f_{A}(a_{0})=(-1)^{\sum_{i=1}^{(N-1)/2}i}\zeta^{\sum_{i=1}^{(N-1)/2}i^2}=(-1)^{(N^2-1)/8}\zeta^{(N-1)(N+1)N/24},
$$
which is a $6$-th root of unity.

Suppose $N$ is prime to $3$. Then $g(a)$ is an $N$-th root of unity, up to sign. Since the group $H_{N}$ acts transitively on the cups above $\infty$, and its action is $\Q(\zeta)$-rational, we deduce that all the cusps above $\infty$ are defined over $\Q(\zeta)$. A similar reasoning apply to the cusps above $0$, and above $1$. 

A similar reasoning applies when $3$ divides $N$. Indeed $g(a)$ is a $3N$- th root of unity, up to sign.

\end{proof}

The cusps of $F'_N$ above the cusp $\infty$ (resp. $0$, resp. $1$) of $X(2)$ coincide with the classes $\Ga\backslash\Ga(2)\infty$ (resp.  $\Ga\backslash\Ga(2)0$, resp.  $\Ga\backslash\Ga(2)1$), which in turn can be identified with the double classes  $\Ga\backslash\Ga(2)/A^{\Z}$ (resp. $\Ga\backslash\Ga(2)/B^{\Z}$, resp. $\Ga\backslash\Ga(2)/(AB^{-1})^{\Z}$).

\subsection{About the Manin-Drinfeld principle for $F'_3$}
Recall that $g_{3}=1$ and observe that $F'_{3}$ has 27 cusps. Fix one cusp $P_{0}$ of $F'_{3}$, which becomes thus an elliptic curve $(F'_{3}, P_{0})$. Since a cyclic group of order $3$ acts on $F'_{3}$ and stabilizes $P_{0}$, the elliptic curve $(F'_{3}, P_{0})$ admits an automorphism of order $3$. Thus the $j$-invariant of $(F'_{3}, P_{0})$ is $0$.
\begin{prop}
Divisors supported on the cusps of $F'_{3}$ above $\infty$ (resp. $0$, resp. $1$) are of order dividing $3$ in the Jacobian of $F'_{3}$. Furthermore, cuspidal divisors of degree $0$ are torsion in the Jacobian of $F'_3$, and of order dividing $9$.
\end{prop}
\begin{proof}
Let $X$ be a compact connected Riemann surface of genus $1$. Let $J$ be the Jacobian of $X$. Recall the exact sequence
$$
0\rightarrow J({\bf C})\rightarrow {\rm Aut}(X)\rightarrow {\rm Aut}(J)\rightarrow 0,
$$
where ${\rm Aut}$ denotes the automorphisms  over ${\bf C}$. The first map associates to the class of a divisor $D$ the translation by $D$ in $X$.

For $X=F'_{3}$, the group ${\rm Aut}(J)$ is cyclic of order $6$. One gets a group homomorphism $H_{{\bf Z}/3{\bf Z}}\rightarrow {\rm Aut}(J)$, whose kernel is of order $9$, and therefore isomorphic to $({\bf Z}/3{\bf Z})^{2}$. 
Since this kernel identifies to a subgroup of  the one dimensional complex torus $J({\bf C})$, the latter subgroup is $J({\bf C})[3]$.  We have proved that the orbit of any cups $Q$ by $H_{{\bf Z}/3{\bf Z}}$ contains $Q+J({\bf C})[3]$. But these sets are both of cardinality $9$, and are therefore equal. Since $H_{{\bf Z}/3{\bf Z}}$ acts transitively on the cusps above $\infty$ (resp. $0$, resp. $1$), the first statement of the proposition is proved.

About the second statement, it is sufficient to prove this for a divisor of the form $(\alpha)-(\beta)$ where $\alpha$ and $\beta$ are cusps of $F'_3$ not above the same point of $\{0,1,\infty\}$. Without loss of generality, say they are above $0$ and $\infty$ respectively. Let $a$ and $b$ be the cusps of $F_3$ below $\alpha$ and $\beta$ respectively.

We have 
$$
3((\alpha)-(\beta))=(3(\alpha)-\sum_{\alpha'}(\alpha'))+(\sum_{\alpha'}(\alpha')-\sum_{\beta'}(\beta'))+(\sum_{\beta'}(\beta')-3(\beta)),
$$
where $\alpha'$ (resp. $\beta'$) runs through the cusps of $F'_3$ above $a$ (resp. $b$). By the first statement of the proposition and the torsion properties of the cuspidal subgroup of the $F_3$, each of the three terms of the right-hand side is of order dividing $3$. 
\end{proof} 

Recall that the dessin for $X'_{N}$ is a graph with the following additional structure: the vertices are bicolored (white and black) and the of set edges attached to any given vertex are endowed with a cyclic ordering (a transitive action of $\Z$). 
The vertices are the cusps of $X'_{N}$ above $0$ and $\infty$. The edges form the coset $\Phi'_{N}\backslash\Ga(2)\simeq H_{\Z/N\Z}$, which is in bijection with $(\Z/N\Z)^{3}$. the edge associated with $\Phi'_{N}g$ has extremities  $\Phi'_{N}g0$ and $\Phi'_{N}g\infty$. 
The cyclic ordering of the edges attached to the vertices $\Phi'_{N}g0$ (resp. $\Phi'_{N}g\infty$ ) is given by the action of $B$ (resp. $A^{-1}$). To sum up the dessin can be drawn on $X'_{N}$.

To be more concrete, each edge is in bijection with $\Z/N\Z\times\Z/N'\Z\times\Z/N\Z$, via the map $\Phi'_{N}A^{a}C^{c}B^{b}\mapsto (a,c,b)$.
The edge thus labeled $(a,c,b)$ is connected to the edge labeled $(a,c,b+1)$ via a black vertex (cusp above $0$) and the edge labeled $(a,c,b)$ is connected to the edge labeled $(a-1,c-ab,b)$ via a white vertex (cusp above $\infty$).
The line segments represent the arcs on $X'_3$ above the geodesic arc from $0$ to $\infty$ in the upper half-plane. 

We illustrate all this for $N=3$. In that case, the genus of $X'_{3}$ is equal to $1$. According to these rules, the drawing (dessin) for $X_3'$ is given as follows.
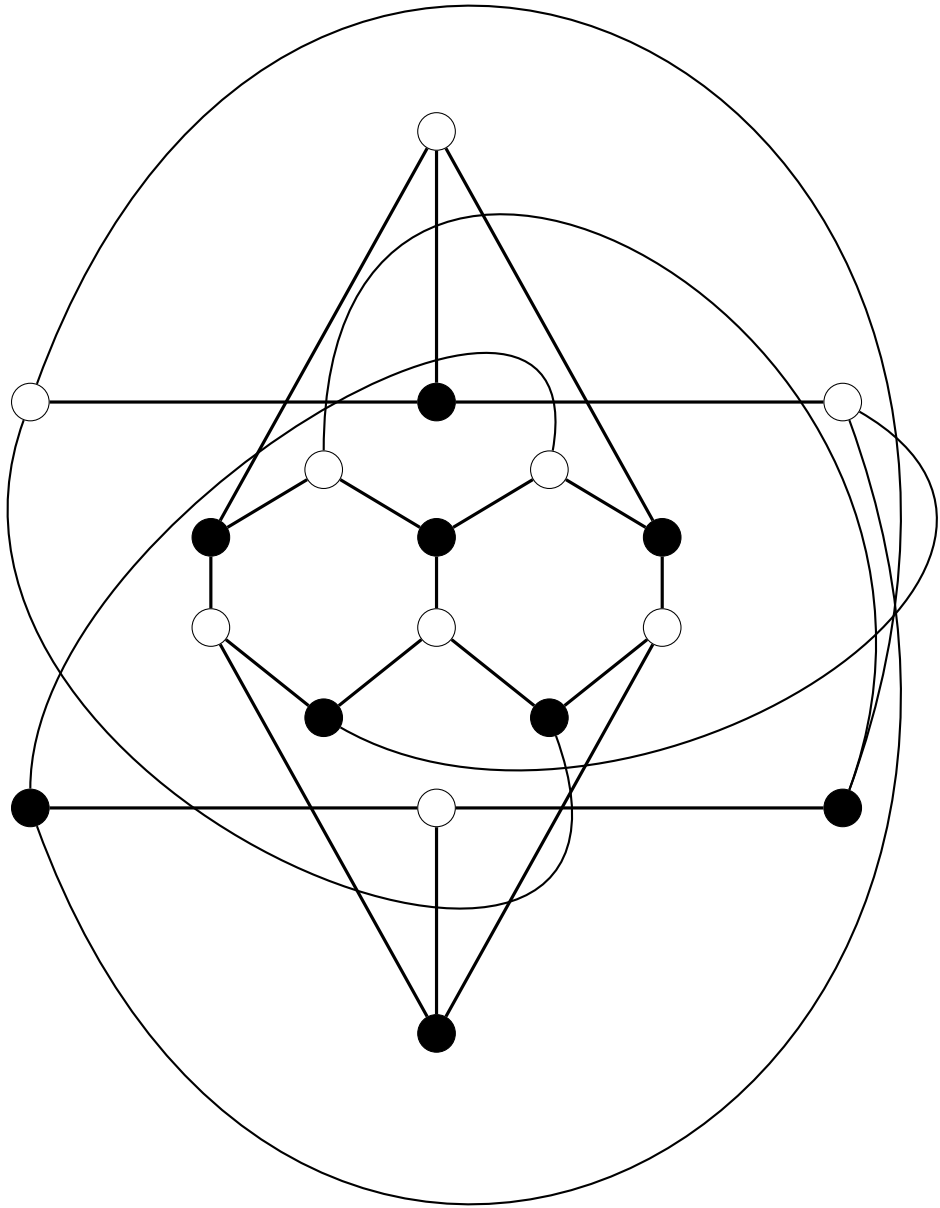
\begin{center}
\begin{tikzpicture}[mynode/.style={font=\color{red}\sffamily,circle,inner sep=1pt,minimum size=.5cm}, scale=0.6]
\node[circle,draw,radius=1em][mynode=black,fill=black] (v1) at (1,-8,0) {};
\node[circle,draw,radius=1em][mynode=black,fill=black] (v2) at (-8,-3,0) { };
\node[circle,draw,radius=.01em][mynode=black,fill=white] (v3) at (1,-3,0) {};
\node[circle,draw,radius=.01em][mynode=black,fill=black] (v4) at (10,-3,0) {};
\node[circle,draw,radius=.01em][mynode=black,fill=black] (v5) at (-1.5,-1,0) {};
\node[circle,draw,radius=.01em][mynode=black,fill=black] (v6) at (3.5,-1,0) {};
\node[circle,draw,radius=.01em][mynode=black,fill=white] (v7) at (-4,1,0) {};
\node[circle,draw,radius=.01em][mynode=black,fill=white] (v8) at (1,1,0) {};
\node[circle,draw,radius=.01em][mynode=black,fill=white] (v9) at (6,1,0) {};
\node[circle,draw,radius=.01em][mynode=black,fill=black] (v10) at (-4,3,0) {};
\node[circle,draw,radius=.01em][mynode=black,fill=black] (v11) at (1,3,0) {};
\node[circle,draw,radius=.01em][mynode=black,fill=black] (v12) at (6,3,0) {};
\node[circle,draw,radius=.01em][mynode=black,fill=white] (v13) at (-1.5,4.5,0) {};
\node[circle,draw,radius=.01em][mynode=black,fill=white] (v14) at (3.5,4.5,0) {};
\node[circle,draw,radius=.01em][mynode=black,fill=white] (v15) at (-8,6,0) {};
\node[circle,draw,radius=.01em][mynode=black,fill=black] (v16) at (1,6,0) {};
\node[circle,draw,radius=.01em][mynode=black,fill=white] (v17) at (10,6,0) {};
\node[circle,draw,radius=.01em][mynode=black,fill=white] (v18) at (1,12,0) {};
\draw [black,thick]   (v15) to[out=70,in=70, distance=17cm ] (v4);
\draw [black,thick]   (v17) to[out=-70,in=-70, distance=17cm ] (v2);
\draw[black,very thick] (v2)--(v3);
\draw[black,very thick] (v4)--(v3);
\draw[black,very thick] (v15)--(v16);
\draw[black,very thick] (v17)--(v16);
\draw[black,very thick] (v10)--(v18);
\draw[black,very thick] (v10)--(v7);
\draw[black,very thick] (v16)--(v18);
\draw[black,very thick] (v12)--(v18);
\draw[black,very thick] (v12)--(v9);
\draw[black,very thick] (v1)--(v7);
\draw[black,very thick] (v1)--(v3);
\draw[black,very thick] (v1)--(v9);
\draw[black,very thick] (v5)--(v7);
\draw[black,very thick] (v5)--(v8);
\draw[black,very thick] (v6)--(v8);
\draw[black,very thick] (v11)--(v8);
\draw[black,very thick] (v13)--(v10);
\draw[black,very thick] (v13)--(v11);
\draw[black,very thick] (v14)--(v11);
\draw[black,very thick] (v14)--(v12);
\draw[black,very thick] (v6)--(v9);
\draw [black,thick]   (v2) to[out=90,in=80, distance=6cm ] (v14);
\draw [black,thick]   (v15) to[out=-110,in=-70, distance=9cm ] (v6);
\draw [black,thick]   (v13) to[out=90,in=70, distance=11cm ] (v4);
\draw [black,thick]   (v17) to[out=-30,in=-30, distance=7cm ] (v5);
\end{tikzpicture}
\end{center}

As the group $\Phi_{3}$ is not a congruence subgroup \cite{MR1091611}, $\Phi'_{3}$ is {\it a fortiori} not a congruence subgroup. The latter fact can be derived alternately from Wolfart's criterion and an examination of the dessin.
Indeed, the width of the each cusps is equal to $6$ and $[\Ga(2):\Phi'_3]=27$.  
Wolfart's criterion, to check that $\Phi'_3$ is a congruence subgroup or not it is 
enough to check $\Ga(6) \subset \Phi'_3 \subset \Ga(2)$. However, $[\Ga(2):\Ga(6)]=144$ is not divisible by the index  $27$.

\subsection{About the failure of the Manin-Drinfeld principle for $F'_5$}
Suppose $N=5$. We show that the Manin-Drinfeld theorem fails for the Heisenberg covering ${\mathcal{F}}'_5$ of ${\mathcal{F}}_5$. 
Consider the scheme ${\mathcal{C}}_5$ over ${\rm \Spec}({\bf Z}[\mu_{5},1/5])$ given by the system of equations 
$$
T_\zeta^5=\frac{(-Y+\zeta)(-Y+\zeta^2)^2}{(-Y+\zeta^{-1})(-Y+\zeta^{-2})^2},
$$
where $\zeta$ runs through the primitive $5$-th roots of unity in $\Q(\mu_{5})$. It is smooth, as can be shown by applying the Jacobian criterion. 
One has an obvious morphism of schemes ${\mathcal{F}}'_5\rightarrow {\mathcal{C}}_5$.

Denote by $C_{5}$ the generic fiber of ${\mathcal{C}}_5$. Over $\C$, $C_5$ identifies to a modular curve as follows. 
Consider the morphism $\bar\Ga(2)\rightarrow H_{5,5,5}$ and the inverse image $\Gamma'$ of the subgroup of $H_{5,5,5}$ generated by $B$. 
Then $\Gamma'$ defines a corresponding modular curve isomorphic to the Riemann surface $C_5(\C)$. 

The curve $C_{5}$ possesses $19$ cusps, given by the following planar coordinates $(Y,T_{\zeta})$ : $(0,\epsilon)$, $(\infty, \epsilon)$, $(1, -\delta)$, $(\zeta,0)$, $(\zeta^2,0)$, $(\zeta^{-1},\infty)$, $(\zeta^{-2},\infty)$, where $\epsilon$ and $\delta$ run through the $5$-th roots of unity.  
Set $T=T_{\zeta}$. The function fields of $F'_{5}$ and $C_{5}$ are ${\bf Q}(\zeta,X,Y,T)$ and ${\bf Q}(\zeta,Y,T)$ respectively.

The obvious morphism $\pi$ : $F'_5\rightarrow C_5$ sends the cusps of $F'_5$ to the cusps of $C_5$. We show that $C_5$ does not satisfy the Manin-Drinfeld principle. 
Since ${\mathcal{C}}_5$ is smooth, the Jacobian of $C_5$ extends to an abelian scheme ${\mathcal{J}}_5$
over ${\rm \Spec}({\bf Z}[\mu_5,1/5])$.

\begin{prop}
There exists a divisor of infinite order supported on the cusps of $C_{5}$.
\end{prop}
\begin{proof}
We suppose that all cuspidal divisors are torsion in $C_{5}$. 
Our proof is organized around the following calculation.
One has 
\begin{align}
\label{fivroot}
T^5+1=\frac{(1-Y)(2Y^2+(2-2(\zeta^2+\zeta^{-2})-\zeta-\zeta^{-1})Y+2)}{(-Y+\zeta^{-1})(-Y+\zeta^{-2})^2}.
\end{align}
Let $y_1$ and $y_2$ be the roots of the polynomial $2Y^2+(2-2(\zeta^2+\zeta^{-2})-\zeta-\zeta^{-1})Y+2$. 
Let $\zeta_1$ be a $5$-th root of unity in ${\bf Z}[\mu_5]$. Consider the function $T+\zeta_1$, which divides $T^5+1$. 
The divisor of the function $T+\zeta_1$ is (in terms of planar coordinates for $(Y,T)$) : $(1,-\zeta_1)+(y_1,-\zeta_1)+(y_2,-\zeta_1)-D$, where $D=(\zeta^{-1},\infty)+2(\zeta^{-2},\infty)$.
Apparently fortuitously, this divisor is cuspidal in the fibers at $11$ and at $2$ of $C_{5}$.

\begin{lemma}
Let $\zeta_1$ and $\zeta_2$ be distinct primitive $5$-th roots of unity in $\Z[\mu_5]$. 
The divisors $3(1,\zeta_1)-3(1,\zeta_2)$ and $3(1,\zeta_1)-D$ are principal in any fiber above $11$ of $C_5$. 
\end{lemma}
\begin{proof}
In characteristic $11$, the polynomial $2Y^2+(2-2(\zeta^2+\zeta^{-2})-\zeta-\zeta^{-1})Y+2$ has the providential property of having a double root equal to $1$. Therefore the divisor of the function $T+\zeta_1$ over $\F_{11}$ is cuspidal and equal to $3(1,\zeta_1)-D$. 
It follows that the function  $(T+\zeta_1)/(T+\zeta_2)$ has divisor $3(1,\zeta_1)-3(1,\zeta_2)$.
\end{proof}

We return to the proof of the proposition. 
By our Manin-Drinfeld assumption in $C_{5}$, the divisor $3(1,\zeta_1)-D$ is torsion in the Jacobian of $C_{5}$. 
It extends to a torsion point of ${\mathcal{J}}_5$, whose order is determined in any special fiber at a prime $\pi$ of residual characteristic $p$, provided $p-1>e$, where $e$ is the ramification index at $\pi$ of the extension ${\Q}(\mu_5)|{\Q}$
\cite{MR482230}. This applies for any prime except perhaps $p=2$ and $p=5$.
The calculation for $p=11$ ensures that the divisors $3(1,\zeta_1)-3(1,\zeta_2)$ and $3(1,\zeta_1)-D$  are principal in $C_{5}$.

Therefore those divisors are principal in any special fiber of ${\mathcal{C}}_5$.
Consider any fiber $\bar{C}$ above $2$ of $C_{5}$. In that fiber, the function $T-\zeta=T+\zeta$ has divisor, in view of Equation~\ref{fivroot}, $(0,\zeta)+(1,\zeta)+(\infty,\zeta)-D_{\zeta}$, which is principal.
Thus, by taking $\zeta_{1}=\zeta$, the divisor 
\[
(3(1,\zeta)-D_{\zeta})-((0,\zeta)+(1,\zeta)+(\infty,\zeta)-D_{\zeta})=2(1,\zeta)-(0,\zeta)-(1,\zeta)
\] is principal in $\bar C$. Thus there exists $f$ : ${\bar C}\rightarrow  \sP^{1}$ of degree $2$. So $\bar{C}$ is hyperelliptic.
The principality of the divisor $3(1,\zeta_1)-3(1,\zeta_2)$ ensures that there is a degree $3$ morphism $\bar C\rightarrow \sP^{1}$. By Castelnuovo-Severi type inequalities, this imposes that the genus of $\bar {C}$ is $\le2$. But the genus of $\bar C$ equals the genus of $C_{5}$; it is equal to $6$ and we have reached a contradiction. 
 \end{proof}
 
\bibliographystyle{Crelle}
\bibliography{DEBARGHAMERELEisenstein.bib}

\def\cprime{$'$}
\begin{thebibliography}{10}

\bibitem{BanerjeeMerel}
D.~Banerjee and L.~Merel, \emph{The Eisenstein cycles and ManinDrinfeld
  properties}, Arxiv  (2022)\hskip -.1cm.

\bibitem{Curilla}
C.~Curilla, Regular models of Fermat's curves and applications to Arakelov
  theory, ProQuest LLC, Ann Arbor, MI (2010). Thesis (Ph.D.)--Yale University.

\bibitem{MR1012168}
P.~Deligne, \emph{Le groupe fondamental de la droite projective moins trois
  points}, in Galois groups over Q (Berkeley, CA, 1987), Vol.~16 of \emph{Math.
  Sci. Res. Inst. Publ.}, 79--297, Springer, New York (1989).

\bibitem{MR3951413}
O.~Ejder, \emph{Modular symbols for {F}ermat curves}, Proc. Amer. Math. Soc.
  \textbf{147} (2019), no.~6,  2305--2319.

\bibitem{MR0103215}
M.~Hall, Jr., The theory of groups, The Macmillan Company, New York, N.Y.
  (1959).

\bibitem{MR681120}
S.~Lang, Introduction to algebraic and abelian functions, Vol.~89 of
  \emph{Graduate Texts in Mathematics}, Springer-Verlag, New York-Berlin,
  second edition (1982), ISBN 0-387-90710-6.

\bibitem{MR1917232}
Q.~Liu, Algebraic geometry and arithmetic curves, Vol.~6 of \emph{Oxford
  Graduate Texts in Mathematics}, Oxford University Press, Oxford (2002), ISBN
  0-19-850284-2. Translated from the French by Reinie Ern\'{e}, Oxford Science
  Publications.

\bibitem{MR0314846}
J.~I. Manin, \emph{Parabolic points and zeta functions of modular curves}, Izv.
  Akad. Nauk SSSR Ser. Mat. \textbf{36} (1972) 19--66.

\bibitem{MR488287}
B.~Mazur, \emph{Modular curves and the {E}isenstein ideal}, Inst. Hautes
  \'Etudes Sci. Publ. Math.  (1977), no.~47,  33--186 (1978).

\bibitem{MR482230}
---{}---{}---, \emph{Rational isogenies of prime degree (with an appendix by
  {D}. {G}oldfeld)}, Invent. Math. \textbf{44} (1978), no.~2,  129--162.

\bibitem{MR1405312}
L.~Merel, \emph{L'accouplement de {W}eil entre le sous-groupe de {S}himura et
  le sous-groupe cuspidal de {$J_0(p)$}}, J. Reine Angew. Math. \textbf{477}
  (1996) 71--115.

\bibitem{MR983619}
V.~K. Murty and D.~Ramakrishnan, \emph{The {M}anin-{D}rinfeld theorem and
  {R}amanujan sums}, Proc. Indian Acad. Sci. Math. Sci. \textbf{97} (1987), no.
  1-3,  251--262 (1988).

\bibitem{MR1091611}
R.~Phillips and P.~Sarnak, \emph{The spectrum of {F}ermat curves}, Geom. Funct.
  Anal. \textbf{1} (1991), no.~1,  80--146.

\bibitem{Posingies}
A.~E. Posingies, Belyi pairs and scattering constants, Ph.D. thesis,
  Humboldt-University Berlin, Berlin (2010).

\bibitem{MR2626317}
D.~E. Rohrlich, Modular functions and the Fermat's curves, ProQuest LLC, Ann
  Arbor, MI (1976). Thesis (Ph.D.)--Yale University.

\bibitem{MR0441978}
---{}---{}---, \emph{Points at infinity on the {F}ermat curves}, Invent. Math.
  \textbf{39} (1977), no.~2,  95--127.

\bibitem{MR582434}
J.~V\'{e}lu, \emph{Le groupe cuspidal des courbes de {F}ermat}, in
  S\'{e}minaire {D}elange-{P}isot-{P}oitou, 20e ann\'{e}e: 1978/1979.
  {T}h\'{e}orie des nombres, {F}asc. 2 ({F}rench), Exp. No. 28, 11,
  Secr\'{e}tariat Math., Paris (1980).

\bibitem{MR0563089}
V.~V. \v{S}okurov, \emph{Holomorphic differential forms of highest degree on
  {K}uga's modular varieties}, Mat. Sb. (N.S.) \textbf{101(143)} (1976), no.~1,
   131--157, 160.

\bibitem{MR582162}
---{}---{}---, \emph{Shimura integrals of cusp forms}, Izv. Akad. Nauk SSSR
  Ser. Mat. \textbf{44} (1980), no.~3,  670--718, 720.

\bibitem{MR571104}
---{}---{}---, \emph{A study of the homology of {K}uga varieties}, Izv. Akad.
  Nauk SSSR Ser. Mat. \textbf{44} (1980), no.~2,  443--464, 480.

\end{thebibliography}
\end{document}